\documentclass[11pt]{amsart}
\usepackage[margin=3cm]{geometry}
\usepackage{bm}
\usepackage{amsthm}
\usepackage{amssymb}
\usepackage{mathrsfs}
\usepackage{hyperref,xcolor}
\usepackage[foot]{amsaddr}

\newcommand{\R}{\mathbb{R}}
\newcommand{\N}{\mathbb{N}}
\newcommand{\Z}{\mathbb{Z}}
\newcommand{\e}{\epsilon}
\newcommand{\hc}{\mathcal{H}}
\newcommand{\vphi}{\varphi}

\DeclareMathOperator{\crit}{crit}

\DeclareMathOperator{\ind}{ind}

\newtheorem{thm}{Theorem}
\newtheorem*{thm*}{Theorem}
\newtheorem*{namedthm}{\namedthmname}
\newtheorem{lem}{Lemma}
\newtheorem*{lem*}{Lemma}
\newtheorem{prop}{Proposition}

\newtheorem*{coro*}{Corollary}

\theoremstyle{definition}
\newtheorem*{defi}{Definition}

\newtheorem*{rmk}{Remark}

\newcounter{namedthm}
\makeatletter
\newenvironment{named}[1]
  {\def\namedthmname{#1}%
   \refstepcounter{namedthm}%
   \namedthm\def\@currentlabel{#1}}
  {\endnamedthm}
\makeatother

\title[Allen-Cahn flow connecting minimal surfaces in the 3-sphere]{Morse theory for the Allen-Cahn functional}
\author{Jingwen Chen\textsuperscript{1}, Pedro Gaspar\textsuperscript{2}}
\address{\parbox{\linewidth}{
\textsuperscript{1} Department of Mathematics, University of Pennsylvania, \\
David Rittenhouse Lab,
209 South 33rd Street,
Philadelphia, PA 19104 \\ \textsuperscript{2} Facultad de Matem\'aticas, Pontificia Universidad Cat\'olica de Chile \\ Avenida Vicuña Mackenna 4860, Santiago, Chile \smallskip}}
\email{jingwch@sas.upenn.edu, pedro.gaspar@mat.uc.cl}

\begin{document}

\begin{abstract}
In this article, we use Morse-theoretic techniques to construct connections between low energy critical submanifolds of the Allen-Cahn energy functional in the $3$-sphere via the negative gradient flow.
\end{abstract}

\maketitle
\section{Introduction}

Morse theory, which allows for the analysis of the topology of a manifold by studying differentiable functions on that manifold, was extensively investigated in infinite-dimensional settings by various mathematicians -- Smale \cite{smale1961gradient, smale1967differentiable}, Palais \cite{Palais}, Bott \cite{bott1982lectures, bott1988morse}, Conley-Zehnder \cite{conley1984morse}, Floer \cite{floer1989witten}, Chang\cite{chang1993infinite}, and many others. Abbondandolo-Majer \cite{abbondandolo2001morse} introduced a comprehensive Morse homology theory for functionals on a Hilbert space. Their approach revolves around the idea of examining the negative gradient flows of functionals that connect critical points, which are then utilized to construct the Morse complex. They also address a transversality condition between unstable and stable manifolds, a concept that has substantial applications in the fields of differential geometry, topology, and dynamical systems. Some applications of infinite-dimensional Morse theory include the construction of closed geodesics -- see the survey \cite{Geodesics} and the references therein -- and the solution to the triangulation conjecture in topology, by C. Manolescu \cite{Triangulation}.\medskip

Motivated by the achievements of Morse Theory and recognizing the deep connections between the Allen-Cahn equation and Mean Curvature Flow, we study Morse theory for the \emph{Allen-Cahn energy functional} in the $3$-sphere, namely:
\[E_{\epsilon}(u)= \int_{M} \left(\frac{\epsilon}{2} |\nabla_{g} u|^2 + \frac{1}{\epsilon} W(u)\right).\]
on the Hilbert space $W^{1,2}(S^3)$, aiming to understand the negative gradient flows that connect critical submanifolds of $E_{\e}$ at low energy levels. Here, $W$ represents a symmetric, nonnegative double-well potential with wells located at $\pm 1$, and a typical example is $W(u) = (1-u^2)^2/4$. These flows correspond to solutions of the \emph{parabolic Allen-Cahn equation}:
    \begin{equation} \label{PAC}
        \partial_t u = \Delta u - {\textstyle \frac{1}{\e^2}}W'(u) \tag{PAC}
    \end{equation}
that connect critical points of $E_\e$, which are stationary solutions of this equation. 

In this article, we study the dynamics of solutions to \eqref{PAC} in $S^3$ and demonstrate that there exist connections between solutions at the second lowest (nontrivial) energy level and solutions at the lowest energy level via a negative gradient flow of the Allen-Cahn functional.\medskip

\subsection*{Mean curvature flow and Morse theory for the area functional} The \emph{Mean Curvature Flow} (MCF) is one of the most studied extrinsic geometric flows and evolves hypersurfaces in their normal direction with a speed equal to the mean curvature at each point. It represents the steepest descent flow for the \emph{area functional} -- that is, the volume of codimension-one submanifolds -- and finds numerous profound applications in fields such as geometry, topology, material science, image processing, and general relativity. The simplest examples of MCFs in the Euclidean space $\R^n$ include the shrinking spheres and cylinders. However, there are comparatively few examples of ancient (and eternal, weak) MCFs in compact manifolds -- we mention here \cite{BIS,CM,HS,MS}.

In the last decades, there has been intense activity in the development of a Morse theory for the area functional using tools from Geometric Measure Theory connected to the existence of minimal hypersurfaces, which are the stationary solutions of the MCF, and with deep consequences in geometry and topology. This program was initiated by F. Almgren and J. Pitts \cite{Pitts} (see also Schoen-Simon \cite{Schoen-Simon}) and was further developed by F. Marques and A. Neves \cite{MarquesNevesWillmore,MarquesNevesIndex}, who proposed a Morse-theoretic program for minimal hypersurfaces. After the contribution of many mathematicians -- we name here Y. Liokumovich \cite{YMN}, A. Song \cite{Song}, and X. Zhou \cite{Zhou}, and refer to the survey \cite{MNMorse} for a more complete list of references -- this program established, for instance, the characterization of low area minimal surfaces in $S^3$ \cite{MarquesNevesWillmore} and the existence of infinitely many minimal hypersurfaces in compact Riemannian manifolds of dimension $3\leq n\leq 7$, with a concrete description of their Morse indexes and their area growth \cite{Song,MarquesNevesIndex}.\smallskip

\subsection*{Regularization by the Allen-Cahn energy} The Allen-Cahn equation serves as a model for phase transition and separation phenomena \cite{AC}. The motion of the diffuse transition region, where $u$ remains bounded away from the minima of $W$, has been extensively studied by several authors. They have established the convergence of this interface, as $\e \downarrow 0$, to a hypersurface that evolves according to the mean curvature flow, albeit in different formulations. In that sense, the parabolic equation \eqref{PAC} can be regarded as a regularization for the MCF. Prominent contributions in this regard include works by \cite{BronsardKohn, Chen, ChenGigaGoto, EvansSonerSouganidis, MottoniSchatzman, RSK, Soner} and their respective references.

The classical variational convergence for \emph{stationary} solutions of \eqref{PAC} was studied in the works of Modica and Mortola \cite{M85,MM}, who proved that the Allen–Cahn energy functional $\Gamma$-converges to the \emph{perimeter functional}, a generalization of the $(n-1)$-dimensional volume defined on the space of sets of finite perimeter. In particular, the interfaces of locally minimizing critical points of $E_\e$ (namely, the sets where these functions are bounded away from $\pm 1$) converge, as $\epsilon \downarrow 0$, to local minimizers of the perimeter and are thus regular away from a singular set of dimension $\leq (n-8)$. 

Later, a general convergence result for families of stationary solutions with uniformly bounded energy and index was developed in the combined work of J. Hutchinson, Y. Tonegawa, N. Wickramasekera, and M. Guaraco \cite{HT, TW, Guaraco}. This result is based on the deep regularity theory for minimal surfaces developed by Wickramasekera \cite{Wickramasekera}. In essence, a sequence of critical points ${u_j}$ of $E_\e$ with $\e = \e_j \downarrow 0$ and bounded energy and Morse index has a subsequence such that its nodal set accumulates on a minimal hypersurface in $M$ that is smooth and embedded away from a closed set of Hausdorff dimension $\leq (n-8)$. We also mention the strong convergence theory in $3$-dimensional manifolds developed by O. Chodosh and C. Mantoulidis \cite{ChodoshMantoulidis}, based on earlier work of K. Wang and J. Wei \cite{WangWei}, which establishes the graphical convergence of the nodal sets to the limit minimal surface under nondegeneracy conditions.

The convergence of solutions of \eqref{PAC} to measure-theoretic solutions of the Mean Curvature Flow (in the sense of Brakke \cite{Brakke}) has been explored by Ilmanen \cite{Ilmanen}, Soner \cite{Soner}, and more recently by Tonegawa \cite{T03}, Mizuno-Tonegawa \cite{MizunoTonegawa}, Pisante-Punzo \cite{PP}, and Sato \cite{Sato}. We also mention here the strong convergence results of Nguyen-Wang \cite{NW1} and the recent results of Fischer-Laux-Simon \cite{fischer2020convergence} based on relative entropy methods.

Regarding the existence of stationary solutions of \eqref{PAC}, this was studied through Morse-theoretic and variational methods in \cite{Smith}, where G. Smith employed Morse theory for the Allen-Cahn energy in H\"older spaces to produce an increasing number of solutions of \eqref{AC} as $\e \downarrow 0$. For a min-max approach, see \cite{Guaraco,GasparGuaraco}.\medskip

We now describe our main contribution. Recall that any least energy nonconstant critical point of the Allen-Cahn energy $E_{\e}$ in $S^3$ vanishes exactly along an equatorial sphere as demonstrated in \cite{CGGM}, where these solutions are called \emph{ground state solutions}. For small $\e$, the critical points of $E_\e$ at the second lowest nontrivial energy level is a solution that vanishes precisely along a Clifford torus. Let $u_\e^{-\infty}$ be a critical point of $E_\e$ in $S^3$ which vanishes precisely along $\{x \in S^3 \mid x_1^2 + x_2^2 = x_3^2 + x_4^2 = 1/2\}$. 

\begin{named}{Main Theorem} \label{main}
For any ground state solution $u_\e^{+\infty}$ of the Allen-Cahn equation, there is an eternal solution $u_\e \colon S^3 \times \R \to \R$ to \eqref{PAC} such that
    \[\| u_\e(t,\cdot) - u_\e^{\pm \infty} \|_{W^{1,2}(S^3)} \to 0, \quad \text{as} \quad t \to \pm \infty. \]\smallskip
\end{named}

Previously in \cite{CG2023}, we constructed a $2$-parameter family of eternal solutions to \eqref{PAC} that join $u^{-\infty}_\e$ to Allen-Cahn approximations of certain equatorial spheres. The Main Theorem of the present article can be seen as an extension of this result. We note nevertheless that the techniques used in this article are substantially different, in which we replace the symmetry argument used in \cite{CG2023} with a topological argument that allows us to obtain further information about the unstable manifold of $u^{-\infty}_\e$ and their forward limit, as we outline next.\smallskip

Due to the symmetry properties of the round sphere, the Allen-Cahn functional is a \emph{Morse-Bott function} along critical points with low energy, as elucidated by the rigidity results in \cite{CGGM, H}. The concept of Morse-Bott functions generalizes Morse functions and allows for non-isolated and degenerate critical points. Formally, a Morse–Bott function is a smooth function defined on a manifold, where the critical set forms a submanifold, and the Hessian remains nondegenerate in the normal direction.

In order to use infinite-dimensional Morse theory and gain deeper insights into the behavior of the Allen-Cahn energy function, we employ a general perturbation technique introduced in \cite{AB, BH}. In essence, we can perturb a Morse-Bott function in a tubular neighborhood of each critical submanifold to produce a one-parameter family of Morse functions. For example, consider the function $f: \R^3 \to \R$, where $f(x,y,z) = x^2 - y^2$. This function qualifies as a Morse-Bott function, with the $z$-axis serving as its critical submanifold. By introducing perturbations near the $z$-axis, we obtain a family of Morse functions given by $f_{\delta} = x^2 - y^2 + \delta z^2$.

Inspired by this argument, we perturb the energy functional $E_{\e}$ in a tubular neighborhood surrounding the low energy critical submanifolds in order to get a one-parameter family of Morse function with controlled energy and Morse indices. The perturbed Allen-Cahn equation is a nonlinear parabolic partial differential equation that can be regarded as an abstract evolution equation in certain function spaces. Our study explores the existence, regularity, continuity, and long-time existence properties of certain solutions to this equation, as described in \cite{Amann,Henry}.

Additionally, employing the Morse index and the structure of the unstable manifold of critical points of the perturbed functional as a key ingredient, we establish the existence of a negative gradient flow connecting any critical point at the second least positive energy level to any critical point at the least positive energy level. More precisely, we use tools from Morse Theory in Hilbert spaces \cite{Palais} and the variational min-max characterization of Clifford tori (as first obtained in \cite{MarquesNevesWillmore}), which allow us to develop a topological contradiction argument to produce an impossible \emph{$5$-sweepout}, in the sense of the min-max construction of critical points of the Allen-Cahn energy studied in \cite{GasparGuaraco}. %We observe that $L^2$-gradient of $f_\delta$ is a quasilinear parabolic equation with a nonlocal perturbation term of high order, which requires us to use the existence and regularity theory from  to study the solutions of this flow.

\subsection*{Potential geometric applications.} In the proof of our main result, we develop a method to study the structure of the unstable manifold and stable manifold of critical points of the Allen-Cahn energy, with many potential geometric applications. Motivated by the Measure-theoretic convergence of the Allen-Cahn flow to the MCF, one of these applications would be to apply a method similar to the one that the authors presented in \cite{CG2023} to construct Brakke flows with prescribed asymptotic limits within a compact manifold, as we briefly describe below.

%Consider a closed Riemannian manifold $(M^n,g)$. \textcolor{red}{incomplete}

In our previous work \cite{CG2023}, we employed the convergence theory of solutions of \eqref{PAC} to Brakke flows to obtain a $2$-parameter family of eternal Brakke flows in $S^3$ which converge backward in time to the Clifford torus and forward in time to an equatorial sphere. These flows are smooth for large positive and negative times and can be thought of as a connection between the unstable manifold of the Clifford torus, as described in \cite{CM}, and the stable manifold of the equatorial spheres. We believe that the Allen-Cahn flows produced in the present work should give rise to the entire space of connections between these minimal surfaces through mean curvature flows. We will explore this application in upcoming work.

We also mention that the methods in the present article could be useful to study variational characterization results of low-energy or low-index solutions of the Allen-Cahn equation in other manifolds, similar to those obtained in \cite{CGGM,H}. This could potentially be used to describe certain min-max minimal hypersurfaces which are obtained as Allen-Cahn limits, motivated by the close connections between the variational theories for these functionals \cite{GasparGuaraco,Dey}.

\subsection*{Organization}
In Section \ref{sec:preliminaries}, we state some results concerning the Allen-Cahn equation, min-max theory, existence theory and dynamics for abstract parabolic differential equations, and the Morse-Bott theory on Hilbert spaces which will be used in the sequel. 

In Section \ref{sec: Heteroclinic orbits for a perturbed flow}, we present the Morse-Bott perturbation argument, examine heteroclinic orbits for the perturbed Allen-Cahn energy, and ultimately utilize these perturbed orbits to construct the desired orbit for the Allen-Cahn energy. 

In Section \ref{sec: topological argument}, we employ a topological argument to construct appropriate eternal solutions for the negative $L^2$-gradient flow of perturbations to the Allen-Cahn energy. This provides us with the orbits for the perturbed Allen-Cahn energy as utilized in the previous section.

\subsection*{Acknowledgements}

We would like to thank Andr\'e Neves for his continued support and for many invaluable discussions and suggestions. We also thank Davi Maximo for his interest in this work. PG was partially supported by ANID (Agencia Nacional de Investigaci\'on y Desarrollo) FONDECYT Iniciaci\'on grant.

\section{Preliminaries} \label{sec:preliminaries}

\subsection{Phase transitions} \label{subsec:doublewell}

This subsection collects some basic facts and results about the \emph{Allen-Cahn} equation. 

\begin{defi}
A function $W \in C^{\infty}(\R)$ is a \emph{(symmetric) double-well potential} if:
\begin{enumerate}
    \item[(1)] $W$ is nonnegative and vanishes precisely at $\pm1$;
    \item[(2)] $W$ satisfies $W'(0) = 0$, $W''(0) \neq 0$, and $sW'(s) < 0$ for $|s| \in (0, 1)$;
    \item[(3)] $W''(\pm1)>0$;
    \item[(4)] $W(s) = W(-s)$ for all $s \in \R$.
\end{enumerate}
\end{defi}

We will also assume that $W$ verifies the following technical \emph{quadratic growth} condition: there exist $C>0$ and $R_0>1$ such that

\begin{enumerate}
    \item[(Q1)] $|W(s)|\geq \frac{1}{C}|s|^2$ for all $|s| \geq R_0$;
%        \item[(Q2)] $|W'(s)| \leq C(1+|s|)$ for all $|s| \geq R_0$;
    \item[(Q2)] $|W''|$ is bounded.
\end{enumerate}

The standard example of a double-well potential is the function $W(t) = \frac{1}{4}(1 - t^2)^2$. Even though this potential does not satisfy the condition (Q2) above, we observe that, by the maximum principle, any nonconstant classical solution of \eqref{PAC} in a compact manifold $M$ defined for $t\geq t_0$ and such that $|u(\cdot,t_0)|\leq 1$ satisfies $|u(\cdot,t)|<1$ for all $t >t_0$. This means, taking into account that our domains are compact manifolds, one can modify such potential $W(t)$ outside of $[-2,2]$ to obtain a new nonlinearity $\tilde W$ that satisfies (Q1)--(Q2) and such that any solution $u$ with $|u|\leq 1$ of \eqref{PAC} for the potential $\tilde W$ also satisfies \eqref{PAC} for the original double-well potential $W$. Compare with the modifications employed in \cite[Section 4]{Guaraco} in order to apply variational methods for the Allen-Cahn energy.% \textcolor{red}{incomplete}

% \begin{enumerate}
%     \item[(A1)] There is $C>0$ and $s \in (2,4)$ such that
%         \[|W'(t)|\leq C(1+|t|^s) \qquad \text{for all} \ t.\]
%     %\item[(A2)] $W'(t)$ is superlinear, that is $\lim_{t \to +\infty}\frac{W'(t)}{t}=+\infty$.
%     \item[(A2)] There exist $R_0>0$ and $\rho \in [0,\frac{1}{2})$ such that
%         \[W(t)-t^2 \geq \rho t W'(t) \qquad \text{for all} \qquad t \geq R_0.\]
% \end{enumerate}

\begin{defi}
Let $(M^n,g)$ be a compact Riemannian manifold. We define the \emph{Allen-Cahn energy} on $M$ by:
\[E_{\epsilon}(u):= \int_{M} \left(\frac{\epsilon}{2} |\nabla_{g} u|^2 + \frac{1}{\epsilon} W(u)\right) d\mu_{g}, \ \ u \in W^{1,2}(M),\]
where $d\mu_g$ is the volume measure with respect to $g$. %Note that this quantity is finite, provided $W(u) \in L^1(M)$.
\end{defi}

We observe that the technical conditions above ensure that $E_\e$ is a $C^3$ functional in $W^{1,2}(M)$ for compact $3$-dimensional manifolds $M$, and that $E_\e$ satisfies the Palais-Smale condition along sequences that have bounded energy -- see \cite{RabinowitzMinmax}, \cite[Lemma 4.1]{CGGM}.\medskip

One can check that $u$ is a critical point of $E_{\epsilon}$ on a closed manifold $(M^n,g)$ if and only if $u$ (weakly) solves the \emph{elliptic Allen-Cahn equation}:
\begin{equation} \label{AC}
\Delta_g u -  {\textstyle \frac{1}{\e^2}}W'(u)=0 \quad \text{on} \ M. \tag{AC}
\end{equation}

We write $\sigma = \int_{-1}^1 \sqrt{W(t)/2} dt$. This is the energy of the \emph{heteroclinic solution} $\mathbb{H}_\e(t)$ of \eqref{AC} on $\R$, that is, the unique bounded solution in $\R$ (modulo translation) such that $\mathbb{H}_\e(t) \to \pm 1$ when $t \to \pm \infty$. We refer to \cite[Section 1.3]{ChodoshMantoulidis} for more information on this one-dimensional solution.

Recall that the \emph{Morse index} of a solution $u$ of \eqref{AC} (as a critical point of $E_\e$), denoted  $\ind_\e(u)$, is the index of the quadratic form given by the second variation of the energy $E_\e$ at $u$, namely
    \[D^2E_\e[u](\phi,\psi) := \int_M \e\langle \nabla \phi, \nabla \psi \rangle + \frac{1}{\e}W''(u)\phi\psi\ d\mu_g, \qquad \text{for} \quad \psi,\phi \in C^\infty(M).\]
Note that $\ind_\e(u)$ is the number of negative eigenvalues of the self-adjoint, second-order linear elliptic operator 
    \[DAC_\e(f) := \Delta f - \frac{W''(u)}{\e^2} f,\]
counted with multiplicity. In particular, $\ind_\e(u)$ is finite (note that we assumed $M$ to be compact). We also recall that $u$ is said to be a \emph{stable} solution if $\ind_\e(u)=0$.

\subsection{Min-max theory for even functionals} \label{subsec:minmax}

We briefly recall the definition of the cohomological $\Z_2$-index (for subsets of a Hilbert space), and the definition of the $p$-widths for an even functional on such spaces; see \cite{FR,GasparGuaraco}. Along this section, we consider a Hilbert space $\mathcal{H}$ and a $C^2$ even functional $f \colon \mathcal{H} \to \R$ that has the origin in $\mathcal{H}$ as an isolated critical point and that satisfies the Palais-Smale property along sequences $\{u_j\} \subset \mathcal{H}$ for which $\{f(u_j)\}$ is bounded. We also assume that $D^2f(u)$ is a Fredholm operator for every critical point $u \in \mathcal{H}$ of $f$.

We say that $A\subset \mathcal{H}$ is a \emph{symmetric} set if $-u \in A$ whenever $u \in A$. For every symmetric $A \subset \mathcal{H}$ with $0 \notin A$, denote by $q_A \colon A \to \tilde A$ the quotient map under the action of the antipodal map $u \in A \mapsto -u \in A$. If $S^\infty$ and $\R P^\infty$ denote the infinite dimensional sphere and real projective space (with the final topology given by the inclusions $S^n \to S^{n+1}$ and $\R P^n \to \R P^{n+1}$), then there exists a \emph{classifying map} for $A$, namely a continuous map $f_A\colon A \to S^\infty$ which is equivariant with respect to the antipodal maps, and hence induces a continuous $\tilde f_A\colon \tilde A \to \R P^\infty$ such that $\tilde f_A \circ q_A = \mathrm{pr} \circ f$, where $\mathrm{pr}\colon S^\infty \to \R P^\infty$ is the quotient map. We recall that the \emph{cohomological $\Z_2$ index of $A$} is defined as
    \[\ind_{\Z_2}(A) = \sup \{k \in \N \mid \tilde f_A^*(w^{k-1}) \neq 0 \in H^{k-1}(\tilde A,\Z_2)\},\]
where $w \in H^1(\R P^\infty,\Z_2)$ is the generator of the cohomology of $\R P^\infty$, and we use Alexander-Spanier cohomology groups with coefficients in $\Z_2$. Moreover, we adopt the convention that $\ind_{\Z_2}(\emptyset)=0$ and that $\ind_{\Z_2}(A) = \infty$ whenever $A$ is symmetric and $0 \in A$.

For each $p \in \N$, we consider the families
    \[\mathcal{F}_p = \{A \subset \mathcal{H} \mid A \ \text{is symmetric}, \ \text{and} \ \ind_{\Z_2}(A) \geq p+1\}.\]
For instance, the unit sphere in any $(p+1)$-dimensional subspace of $\mathcal{H}$ is in $\mathcal{F}_p$. We define the $p$-\emph{widths} of $f$ as
    \[\omega_p(f):= \inf_{A \in \mathcal{F}_p} \sup_{u \in A}f(u).\]
By using variational methods (see \cite{Ghoussoub}), one proves that $\omega_p(f)$ are critical values of $f$, and are achieved by critical points $u \in \mathcal{H}$ of $f$ such that $\ind(D^2f(u)) \leq p \leq \ind(D^2f(u)) + \dim\ker D^2f(u)$.

We also observe that if $f,g \colon \hc\to \R$ are even $C^2$ functionals, then
    \[\sup_{u \in A} f(u) \leq \sup_{u \in A}g(u) + \sup_{\hc}|f-g|\]
for any $A \in \mathcal{F}_p$, so that $\omega_p(f) \leq \omega_p(g) + \sup_{\hc}|f-g|$. The following lemma is a useful consequence of this observation:

\begin{lem} \label{lem:convwidth}
Let $\{f_\delta\}$ be a family of even $C^2$ functionals defined on a Hilbert space $\hc$ satisfying the conditions described above, for every small $\delta>0$. Assume that there exists an even $C^2$ functional $f \colon \hc \to \R$ which also satisfies the conditions above and such that $\sup_{H}|f-f_\delta| \to 0$, as $\delta \downarrow 0$. Then $\omega_p(f_\delta) \to \omega_p(f)$ as $\delta \downarrow 0$, for all $p$.
\end{lem}

We now focus on the min-max values $\omega_p(E_\e)$ for the Allen-Cahn energy functional $E_\e$. The convergence theory described in the introduction indicates that these critical values should converge to critical values for the \emph{area functional}, whose critical points are stationary surfaces for the codimension-one MCF, namely, minimal hypersurfaces. We recall that the variational theory for the latter was initiated by F. Almgren, J. Pitts \cite{Pitts}, and motivated many recent outstanding developments; see \cite{MarquesNevesWillmore,MarquesNevesIndex,Zhou,Song,MNMorse} and the references therein. These results explore the \emph{volume spectrum} of a manifold -- a min-max sequence of critical values for the area -- to prove a number of results about the existence and multiplicity of such surfaces.

The convergence of $\{\omega_p(E_\e)\}_p$ to the volume spectrum as $\e \downarrow 0$ was established in \cite{GasparGuaraco,Dey} (where $\omega_p(E_\e)$ are denoted by $c_\e(p)$) and motivates the search of classification results for solutions of \eqref{AC} from a geometric perspective, under curvature or symmetry assumptions. In the result below we summarize the classification of solutions at energy levels corresponding to the first min-max values $\{\omega_p(E_\e)\}_{p=1}^5$ in the round $3$-sphere which follows from the classification of low area stationary varifolds in $S^3$ \cite{MarquesNevesWillmore}, its phase-transitions analog -- the combined results of \cite{CGGM} and \cite{H} --, and the index bounds for limit interfaces \cite{GIndex,ChodoshMantoulidis} (see Lemma 1 in \cite{ChenTesis} for a proof):

\begin{thm} \label{lowenergy}
There exists $\e_0>0$ such that, for any $\e \in (0,\e_0)$, the following holds. Let $u$ be a nonconstant solution of \eqref{AC} with energy $E_\e(u) \leq \omega_5(E_\e)$. Then either
\begin{itemize}
    \item[(i)] If $E_\e(u_\e)<\omega_5(E_\e)$, then $u$ is a radially symmetric solution of Morse index 1, energy $E_\e(u) = \omega_1(E_\e) = \ldots = \omega_4(E_\e)$ and vanishes precisely on an equatorial sphere.
    \item[(ii)] If $E_\e(u_\e)=\omega_5(E_\e)$, then up to an isometry $u$ is a $\mathrm{SO}(2)\times \mathrm{SO}(2)$-symmetric solution of Morse index 5 and vanishes precisely on the Clifford torus $\{x \in S^3 \mid x_1^2 + x_2^2 = x_3^2 + x_4^2 = \frac{1}{2}\}$. 
\end{itemize}
\end{thm}

The solutions of the first type are called \emph{ground state solutions}, as those are precisely the lowest energy nonconstant solutions (and also the lowest energy unstable solutions) of \eqref{AC}. We note that these solutions, as well as the solutions with symmetries of tori, can be constructed by direct minimization and reflection methods, see \cite{CGGM} for details.

%Almgren showed that $\pi_1(Z_n(M; \Z_2)) = \Z_2$ and $\pi_i(Z_n(M;\Z_2)) = 0$ if $i > 1$. Thus, $Z_n(M;\Z_2)$ is weakly homotopic equivalent to $\R P^{\infty}$. We have that 
%\[H^k(Z_n(M;\Z_2))=\Z_2\] for all $k \in \N$ with generator $\overline{\lambda}^k$.

%We are interested in studying maps $\Phi: X \to Z_n(M;\Z_2)$ whose image detects $\overline{\lambda}^p$ for some $p \in \N$.

%Given a simplicial complex $X$, a continuous map $\Phi: X \to Z_n(M;\Z_2)$ is called a $p-$sweepout if $\Phi^*(\overline{\lambda}^p) \neq 0$ in $H^p(X;\Z_2)$.

%Suppose $X$ is a topological space and $\Phi: X \to Z_n(M)$ is a continuous function. Consider
%\[[\Phi] = \{\Psi: X \to Z_n(M): \Psi \ \rm{homotopic\ to}\ \Phi \ \rm{relative\ to} \ \partial X.\}\]

%Note that if $\Psi \in [\Phi]$ then $\Phi = \Psi$ on $\partial X$. To the homotopy class $\Phi$ we associate the number, called the width,
%\[L([\Phi]) = \inf\limits_{\Psi \in [\Phi]} \sup\limits_{x \in X} \rm{vol}(\Psi(x)).\]

%To conclude this subsection, we briefly describe the results of the low-energy solutions of \eqref{PAC} constructed in \cite{CG2023}.  \textcolor{red}{incomplete, already mention in prop 2?}

%\textcolor{blue}{Place the existence result from Caju-Gaspar here?}

\subsection{Existence theory and dynamics for a class of abstract differential equations} For the sake of completeness, we state some results about the existence and continuous dependence on initial conditions for solutions of a class of abstract parabolic partial differential equations on a compact Riemannian $3$-manifold $(M,g)$. We refer to the results of Sections 7 and 8 of \cite{Amann} (see also \cite{Henry}, for the case of semilinear equations, or \cite{Lunardi}). We will consider initial value problems associated with abstract quasilinear equations in fractional Sobolev spaces $X^\alpha=W^{2\alpha,2}(M)$, and write $X=L^2(M)$. We fix once and for all $\alpha_0 \in \left(\frac{3}{4},1\right)$ so that, by the Sobolev embedding for compact $3$-manifolds, the space $X^\alpha$ continuously embeds into $L^\infty(M) \cap W^{1,2}(M)$ for any $\alpha \in [\alpha_0,1)$.

\begin{rmk}
Here, we utilize the Sobolev embedding theorems for compact manifolds, as detailed in references such as \cite{Aubin} or \cite{LeeParker} (Theorem $2.2$). These theorems establish that, in a certain sense, the Sobolev inequality holds with a consistent constant across all compact manifolds $M$.
\end{rmk}

Consider the differential equation
    \begin{equation} \label{abstract}
        \frac{\mathrm{d}u}{\mathrm{d}t}-(1+\delta\theta(u))\Delta u = F(u,\delta),
    \end{equation}
%\textcolor{blue}{Should this $\theta$ be changed to $\rho$?}
where $F\colon W^{1,2}(M) \times [0,\delta_0) \to W^{1,2}(M)$ is a locally Lipschitz map, $\theta\colon W^{1,2}(M)\to \R$ is a perturbation term with bounded $\|\theta\|_{C^2(W^{1,2}(M))}$ and $\delta_0 >0$ is sufficiently small.

%\textcolor{blue}{$X_{\alpha}$ or $X^{\alpha}$, we may need to explan what is $X^{\alpha}$, like $W^{2,2\alpha}$. In addition, we may need to modify the notation, like change $u(t; \delta, u_0)$ to $u_{\delta, u_0}(t)$.}

To state an existence result for solutions of the initial value problem associated to \eqref{abstract}, we note that, in general, such solutions are not well-defined backward in time. Nevertheless, under suitable conditions, these equations define a locally Lipschitz \emph{semi-flow} in an open subset of $\R \times X^\alpha$. This is the content of the following existence, regularity, and continuity results.

%We first state an existence and regularity result that follows from Theorems 3.3.3 and 3.5.2 (see also Example 3.6) in \cite{Henry}.

\begin{thm} \label{existence} \cite{Amann} Let $t_0 \in \R$ and $u_0 \in X^\alpha$. There exist $T=T(t_0,u_0)>0$ and a maximal solution $u \in C([t_0,t_0+T),X^\alpha) \cap C^1((t_0,t_0+T),X)$ of \eqref{abstract} with initial condition $u_0$ and which satisfies the following additional properties:
\begin{enumerate} 
    \item[(i)] For all $t \in (t_0,t_0+T)$, we have $u(t) \in X^1=W^{2,2}(M)$
    \item[(ii)] $u \in C^1((t_0,t_0+T), X) \cap C_{loc}^{0,\alpha-1/2}((t_0, t_0+T),W^{1,2}(M))$;
    \item[(iii)] If $\alpha' \in (\alpha, 1]$ and $u([t_1,t_0+T))$ is bounded in $X^{\alpha'}$ for some $t_1 \in (t_0,t_0+T)$, then the solution is global, namely $T=+\infty$.
\end{enumerate}

\end{thm}

We will also need the continuous dependence of solutions with respect to initial conditions and the parameter $\delta\in [0,1)$.

\begin{thm} \label{contdependence} \cite{Amann}
For each $u_0 \in X^\alpha$, let $t^+(\delta,u_0)>0$ denote the maximal time of existence of the solution $u_{\delta,u_0}(t)$, $t \in [0,t^+(\delta,u_0))$, of \eqref{abstract} with $u_{\delta,u_0}(0)=u_0$. Write $\mathcal{D}=\{(t,\delta,v) \in \R_{\geq 0}\times [0,1) \times X^\alpha \mid 0 \leq t < t^+(\delta,v)\}$. Then $\mathcal{D}$ is open in $\R_{\geq 0} \times [0,1) \times X^\alpha$ and the map
    \[\Psi \colon \mathcal{D} \to X^\alpha, \qquad \Psi(t,\delta,v) = u_{\delta,v}(t)\]
is a Lipschitz map and defined a $1$-parameter family of semiflows $(t,v)\mapsto \Psi(t,\delta,v)$ in $X^\alpha$.
\end{thm}

%Finally, we state a useful criterion that ensures that the maximally defined solution of \eqref{abstract} is defined for all times.

%\begin{thm} \label{longtime} \cite[Corollary 3.3.5]{Henry} Under the conditions of Theorem \ref{existence}, let $t_0 \in \R$ and $u_0 \in U$, and suppose that there exists $C>0$ such that $\|F(u(t))\|_{L^2(M)}\leq C(1+\|u(t)\|_{X^\alpha})$ for all $t \in (t_0,t_1)$, where $(t_0,t_1)$ is the maximal interval of existence of the solution $u(t)$. Then the unique solution $u(t)$ through $(t_0,u_0)$ is defined for all $t\geq t_0$, namely $t_1 = +\infty$. $u \in U$.
%\end{thm}

%\textcolor{blue}{Topics about analysis of abstract parabolic differential equations on Hilbert (or Banach) spaces. This includes long-time existence results, existence of stable and unstable manifolds, and anything else about the dynamics of gradient (semi)flows near invariant manifolds (including convergence rates, if we need them).}

\subsection{Morse-Bott and Morse functions on Hilbert spaces} \label{subsec:morse} The classification result for low-energy solutions of the Allen-Cahn equation in $S^3$ can be used to describe the topology of certain sublevel sets of the energy $E_\e$, as well as certain perturbations of this functional, using tools of Morse theory.

First, recall some basic results about \emph{Morse-Bott} functionals following the exposition of Feehan \cite{Feehan} in the (simpler) case of a Hilbert space $\hc$. 

\begin{defi}
For an open subset $X$ of $\hc$ and a $C^2$ function $f\colon X \to \R$, we say that $f$ is \emph{Morse-Bott} at a critical point $x_0 \in X$ if

\begin{enumerate}
    \item[(1)] There exist an open subset $U \subset X$ and a connected submanifold $C \subset U$ such that $U \cap \{x \in X \mid Df(x)=0\} = C$;
    \item[(2)] $T_{x_0}C = \ker D^2f(x_0)$.
\end{enumerate}

We note that, as pointed in \cite{Feehan}, since $C$ is a submanifold of $X$ and $T_{x_0}C = \ker D^2f(x_0)$, this space has a closed complement in $\hc$.

If $f$ is Morse-Bott at any critical point $x_0$ with $f(x_0)\leq c$, we will say that $f$ is a \emph{Morse-Bott function on} $\{f \leq c\}$. The following lemma is central to the study of Morse-Bott functions.
\end{defi}

\begin{lem*}[Morse-Bott Lemma for, see \cite{Feehan}, see Theorem 2.10 and Remark 2.12] \label{morsebott}
Let $X$ be an open subset of a Hilbert space $\hc$, and let $f \in C^{k+2}(X)$, $k \geq 1$. Suppose that $f$ is Morse-Bott at $x_0 \in X$. Then there are open neighborhoods $\mathcal{V} \subset \hc$ of the origin, and $\mathcal{U} \subset X$ of $x_0$, and a $C^k$ diffeomorphism $\Phi \colon \mathcal{V} \to \mathcal{U}$ such that $\Phi(0)=x_0$, $D\Phi(0) = \mathrm{id}_\hc$, and
    \[f(\Phi(y)) = f(x_0) + \frac{1}{2}\left\langle y, Ay \right\rangle, \quad \text{for all} \quad y \in \mathcal{V}, \]
where
    \[A = D^2f(x_0) = D^2(f\circ \Phi)(0).\]
\end{lem*}

We will be mostly interested in the case $D^2f(x_0)$ is represented by a self-adjoint operator with a discrete spectrum
    \[\lambda_1 < \lambda_2 \leq \ldots \leq \lambda_p < \lambda_{p+1} = \ldots = \lambda_{p+k} = 0 < \lambda_{p+k+1} < \ldots,\]
where $p \in \N \cup \{0\}$ is the \emph{Morse index} of $f$ at $x_0$, and $k = \dim \ker D^2f(x_0)$ is the \emph{nullity} of the critical point $x_0$. In this case, we can decompose $\hc$ in terms of the negative eigenspaces $E_1,\ldots,E_p$, and positive eigenspaces $E_+$ of $D^2f(x_0)$, so that
    \[f(\Phi(y)) = f(x_0) + \frac{1}{2}\sum_{i=1}^p \lambda_i \|P_i y\|^2 + \frac{1}{2} \left\langle P_+ y, A (P_+y) \right\rangle, \]
where $P_i \colon \hc \to E_i$ and $P_+ \colon \hc \to E_+$ are the orthogonal projections.\medskip

We also recall the following useful:

\begin{lem} \label{critical}
Let $C \subset X$ be a compact, connected, embedded $C^{k+2}$ submanifold. Denote by $N(C)$ the normal bundle of $C$ in $X$ endowed with the metric induced by $X$. For each $r>0$, let $N_r(C) = \{v \in N(C) \mid \|v\|<r\}$. Given an open $U_0 \subset X$ containing $C$, there exist $r>0$ such that the map $\zeta \colon N(C) \to X$ given by $\zeta(v) = u+v$ for $v \perp T_uC$ in $N_r(C)$ is a diffeomorphism onto a neighborhood $U \subset U_0$ of $C$. Moreover, for every $u \in U$, the derivative of the associated nearest point projection $\Pi \colon U \to C$ at $u$ is the orthogonal projection onto $T_uC$.
\end{lem}

\begin{coro*}
Let $f$ and $X$ be as in the Morse-Bott Lemma, suppose $C \subset \crit f$ is an isolated compact, connected, embedded submanifold of $\crit f$. Then the projection map $\Pi\colon U \to C$ given by the Lemma above satisfies $Df(u)v\neq 0$ for any $u \in U \setminus C$ and any $v \in \ker D\Pi(u)$.
\end{coro*}

Next, we check that the Allen-Cahn energy functional on $S^3$ is Morse-Bott at the first three lowest critical values and describe the critical levels at the first min-max values for $E_\e$. This characterization uses the convergence of the Allen-Cahn widths \cite{GasparGuaraco,Dey} and the description of the first $p$-widths for the area functional; see \cite{Nurser,CPCL}.

\begin{prop} \label{energy-mb}
There exists $\e_1>0$ with the following property. For every $\e \in (0,\e_1)$, the the energy functional $E_\e$ is a Morse-Bott functional on $\{E_\e \leq \omega_5(E_\e)=\omega_6(E_\e)=\omega_7(E_\e)\}$. Moreover, $\{E_\e \leq \omega_5(E_\e)\} \cap \crit E_\e$ has precisely four connected components: 
\begin{enumerate}
    \item the global minimizers $\{1\}$ and $\{-1\}$ (as constant functions),
    \item the set $A_\e = \{E_\e = \omega_1(E_\e) = \cdots = \omega_4(E_\e)\} \cap \crit E_\e$ of ground state solutions, which is a smoothly embedded $3$-sphere in $W^{1,2}(S^3)$.
    \item the set $B_\e = \{E_\e = \omega_5(E_\e)=\omega_6(E_\e)=\omega_7(E_\e)\} \cap \crit E_\e$ of $(\mathrm{SO}(2)\times\mathrm{SO}(2))$-symmetric solutions that vanish on some Clifford torus, which is a smoothly embedded Grassmannian manifold $G(2,4)$ in $W^{1,2}(S^3)$.
\end{enumerate}
\end{prop}

\begin{proof}
We first note that if $\e<\e_0$, where $\e_0$ is given by Theorem \ref{lowenergy}, then the critical points listed in the statement are the only critical points of $E_\e$ which have energy $\leq \omega_5(E_\e)$. Moreover, since $D^2E_\e(\pm 1)$ is nondegenerate, it follows that $E_\e$ is Morse-Bott at these constant critical points, for any $\e>0$. 

The fact that $E_\e$ is Morse-Bott on each set of critical points described in (2) and (3) follows from the rigidity results proved in \cite{CGGM} and \cite{H}, respectively, provided we show that $\omega_5(E_\e)=\omega_6(E_\e)=\omega_7(E_\e)$ for sufficiently small $\e$. Otherwise, there would be a sequence $\e_i \downarrow 0$ such that $\omega_5(E_{\e_i}) < \omega_7(E_{\e_i})$. Let $u_i$ is a sequence of critical points of $E_{\e_i}$ with $E_{\e_i}(u_i)=\omega_7(E_{\e_i})$. By \cite{GasparGuaraco,Dey,Nurser}, we have $\omega_7(E_{\e_i}) \to 2\pi^2$, and thus $u_i$ is an Allen-Cahn approximation of a Clifford torus. Using \cite[Corollary 5]{H}, we conclude that $u_i$ is a symmetric solution that vanishes on a Clifford torus and thus $E_{\e_i}(u_i)=\omega_5(E_i)$, which is a contradiction.

Finally, the regularity of these submanifolds is a consequence of the compactness of these critical manifolds, the symmetry of such solutions, and the following observation. Let $\{\psi_p\}_{p \in N}$ be a family of polynomial functions $\psi_p \colon \R^4 \to \R$ parametrized by a smooth manifold $N$ such that $\psi_p(S^3) \subset [-1,1]$ and let $w\colon [-1,1]\to \R$ be a smooth function. Then $p \mapsto w \circ \psi_p|_{S^3}$ is a smooth immersion of $N$ into $W^{1,2}(S^3)$.
\end{proof}

\begin{rmk}
By the results of \cite{CGGM} (see Theorem 2.9), one can also prove that, in the $n$-dimensional round sphere $S^{n}$, the energy $E_\e$ is Morse-Bott at any critical point of energy $\leq \omega_{n+1}(E_\e)$, for any $\e \in (0,\sqrt{-W''(0)/\lambda_1})$, where $\lambda_1$ is the first positive eigenvalue of the Laplace operator in $S^n$.
\end{rmk}

We also include the following result concerning the asymptotic behavior of low energy solutions of \eqref{PAC} in $S^3$ derived on \cite[Proposition 2]{CG2023}, which is essentially a consequence of the smoothing effect of the heat equation and the classification results above.

%\textcolor{blue}{One important things, we use CG to mention our previous paper, but in the citation, CG is the paper by Caju-Gaspar, CG2023 is our paper, we need to fix this. We may rephrase the proposition below, it seems $u(t)$ can be the constant flow of AC approximation of Clifford torus}
%\textcolor{red}{mention $u_{\e}^{-\infty}$ is the ac approximation of clifford torus in prop2, or say is same as in in main theorem. In addition, should we number and label the main theorem to protrude it?} 

\begin{prop} \label{compact_flow}
Let $\e$ be as in Proposition \ref{energy-mb} and let $\{u(t)\}_{t \in \R}$ be an eternal solution of \eqref{PAC} with $E_\e(u(t_1)) < \omega_5(E_\e)$ for some $t_1 \in \R$. If $\|u_\e(t) - (\pm 1)\|_{W^{1,2}(S^3)}$ is bounded from below, then there exists a ground state solution $\bar u$ such that $\|u(t)-\bar u\|_{W^{1,2}(S^3)} \to 0$ as $t \to +\infty$. 
\end{prop}

%\begin{coro}
%Let $u$ be a solution of \eqref{PAC}, maximally defined on some open interval $J \subset \R$ and such that $|u(0)|<1$. Suppose that $E_\e(u(0)) \leq b_\e$. Then $\sup J = +\infty$ and $u(t)$ converges in $W^{1,2}(S^3)$, as $t \to +\infty$ to one of the solutions of \eqref{AC} described in Proposition \ref{energy-mb}. Furthermore, if $\inf J = -\infty$, then $u(t)$ also converges in $W^{1,2}(S^3)$, as $t \to -\infty$, to one of these solutions.
%\end{coro}

We recall that if $x_0$ is a critical point of $f$ with nondegenerate Hessian $D^2f(x_0)$, we say that $f$ is a \emph{Morse function} at $x_0$. This fundamental concept in Morse theory was studied in the context of Hilbert manifolds by R. Palais in \cite{Palais}, where the classical charaterization of the topological changes in the sublevel sets ${f \leq c}$ as $c$ traverses a critical value containing only nondegenerate critical points was extended to infinite dimensional settings, provided that $f$ satisfies a compactness condition now referred to as the \emph{Palais-Smale} condition.

In order to describe Palais' result, we introduce the following definition. Let $X$ and $Y$ be $C^2$ smooth manifolds (possibly with boundary) modeled on a separable Hilbert space $\mathcal{H}$. For $0 \leq k \leq \infty$, denote by $D^k$ the closed unit ball in a $k$-dimensional subspace of $\mathcal{H}$ and by $S^{k-1}$ its boundary. Let $\phi\colon D^k \times D^\ell \to X$ be a homeomorphism onto a closed subset $h \subset X$. We say that $X$ \emph{arises from} $Y$ \emph{by a $C^r$-attachment $\phi$ of type $(k,\ell)$} if
\begin{enumerate}
    \item[(i)] $X = Y \cup h$;
    \item[(ii)] The restriction of $\phi$ to $S^{k-1} \times D^\ell$ is a $C^r$-diffeomorphism onto $h \cap \partial Y$;
    \item[(iii)] The restriction of $\phi$ to $(D^k \setminus S^{k-1}) \times D^\ell$ is a $C^r$-diffeomorphism onto $X \setminus Y$.
\end{enumerate}
More generally, suppose we have a sequence $Y=Y_0,Y_1,\ldots, Y_s=X$ of $C^r$ manifolds modelled on $\hc$ such that $Y_{i+1}$ arises from $Y_i$ by a $C^r$-attachment $f_i$ of a handle o type $(k_i,\ell_i)$. We say that $X$ is obtained from $Y$ by \emph{disjoint $C^r$-attachments $(\phi_1,\ldots, \phi_s)$ of handles of type $(k_1,\ell_1),\ldots, (k_s,\ell_s)$} if the images of $\phi_i$ are pairwise disjoint.

\begin{thm*} \cite{Palais}
Let $f$ be a $C^{r+2}$, $r\geq 1$, function on a complete Hilbert manifold $X$. Suppose that
\begin{enumerate}
    \item[(1)] $a<b$ are such that $[a,b]$ contains an unique critical value $c$ for $f$;
    \item[(2)] $f$ satisfies the \emph{Palais-Smale} compactness condition on $[a,b]$, namely if $(u_k)\subset f^{-1}[a,b]$ is a sequence such that $\|Df(u_k)\| \to 0$, then $(u_k)$ subconverges in $X$ to a critical point of $f$.
    \item[(3)] $f$ is Morse at any critical point in $f^{-1}(c)=\{p_1,\ldots,p_s\}$.
\end{enumerate}
Let $k_i$ be the Morse index of $D^2f(x_0)$ and let $\ell_i$ be its co-index (that is, the index of $-Df^2(x_0)$). Then $f^b=\{x \in M \mid f(x) \leq b\}$ is $C^r$-diffeomorphic to $f^a = \{x \in X \mid f(x) \leq a\}$ by $s$ disjoint $C^r$-attachments $(\phi_1,\ldots,\phi_s)$ of handles of type $(k_1,\ell_1),\ldots,(k_s,\ell_s)$.

Furthermore, $f^b$ has
    \[f^a \cup \phi_1(D^{k_1} \times \{0\}) \cup \ldots \cup \phi_s(D^{k_s}\times \{0\})\]
as a strong deformation retract. 
\end{thm*}

\begin{rmk}
By tracking back the attachment maps $\phi_i$ in the proof of the main result of \cite{Palais} (see sections \S 11 and \S 12), one sees that if $a$ and $b$ are sufficiently close to $c$, then the images $\phi_i(D^{k_i} \times \{0\})$ can be chosen as the image by a local $C^r$ diffeomorphism of a small disk in the space that corresponds to the negative directions for $f$ in Morse-Bott coordinates. This Remark will be used in the proof of the \ref{main} to express the generators of the relative $k_i$-homology groups of $(f^b,f^a)$ in terms of small closed disks on the local unstable manifold of the (negative) gradient flow of $f$ corresponding to critical points in $f^{-1}(c)$ of index $k_i$.
\end{rmk}

%We first state a useful \emph{pull-tight} (or deformation) Lemma which follows from the classical min-max principle for symmetric functionals, see \cite{Ghoussoub}.

%\begin{lem}[Pull-tight Lemma] \label{pull} Let $p \in \N$, and let $\{A_i\}_{i \in \N} \subset \mathcal{F}_p$ be a minimizing sequence. Then there exists a minimizing sequence $\{A_i^*\} \subset \mathcal{F}_p$ such that $\max_{A^*_i} E_\e \leq \max_{A_i} E_\e$, and
%    \begin{center}
%        if $u \in W^{1,2}(M)$ is such that $E_\e(u) = c_\e(p)$ and $\dist(u,A_i^*) \to 0$, then $E_\e'(u)=0$.
%    \end{center}
%\end{lem}

%A minimizing sequence with the additional property above will be called a \emph{pulled-tight} minimizing sequence. It follows from the results of \cite{GasparGuaraco} that if $\e>0$ satisfies $c_\e(p)<E_\e(0)$, then there is at least one such critical point $u$ of $E_\e$, and it satisfies $E_\e(u) = c_\e(p)$ and $\ind_\e(u) \leq p \leq \ind_\e(u)+ \nul_\e(u)$.

\section{Heteroclinic orbits for a perturbed flow} \label{sec: Heteroclinic orbits for a perturbed flow}

We would like to perturb the Allen-Cahn energy functional in $S^3$ to obtain a Morse function defined on $W^{1,2}(S^3)$ with controlled energy and Morse indices near the critical manifolds given by the solutions at the energy levels $\omega_1(E_\e) = \inf\{E_\e(u) \mid DE_\e(u)=0, E_\e(u)>0\}$ -- that is, the ground state solutions -- and $\omega_5(E_\e)$, which are the solutions that vanish on a Clifford torus. This can be achieved by the following general perturbation argument, which is inspired by the construction of \cite{AB,BH}.

%\textcolor{blue}{maybe change $K$ to another notation as $K,k$ are confusing}

\begin{prop} \label{perturb}
Let $X$ be an open subset of a separable Hilbert space $\hc$, and let $f \in C^{k+2}(X)$ satisfy the Palais-Smale condition. Suppose that $C_1,\ldots, C_N \subset \crit f$ are disjoint compact $C^{k+2}$ submanifolds of $X$, and that $f$ is Morse-Bott with finite index at any point of $C:=C_1 \cup \ldots \cup C_N$. Suppose also that $g_i\colon C_i \to \R$ are $C^{k+2}$ Morse functions, let $g: C \to \R$  be the function such that $g|_{C_i} = g_i$, and let $U = N_r(C) \subset X$ be a tubular neighborhood of $C$. There exist $\delta_1>0$ and $C^{k+2}$ functions $f_\delta \colon X \to R$, for $\delta \in [0,\delta_1)$ such that $f_0=f$ and, for any $\delta \in (0,\delta_1)$:
    \begin{enumerate}
        \item[(i)] $f_\delta|_U$ is a Morse function, and $f=f_\delta$ outside $U$;
        \item[(ii)] there exist a smaller neighborhood $N_{r'}(C)$, $r'<r$ of $C$ such that $f_\delta = f + \delta G$ in $N_{r'}(C)$, where $G \in C^{k+2}(U)$ is an extension of $g$ with $G = g \circ \Pi$ and $\Pi\colon U \to C$ is the nearest point projection associated to $U$. In particular, $DG(p) = Dg(\Pi(p)) \circ D\Pi(p)$ for any $p \in N_{r'}(C)$;
        \item[(iii)] $(\crit f_\delta) \cap U = \bigcup_i (\crit g_i)$, and the Morse index of $f_\delta$ at any $x_0 \in C$ is $\ind(f,x_0) + \ind(g_i,x_0)$
        \item[(iv)] $\|f-f_\delta\|_{C^{k}} < c_0 \delta$, for $c_0>0$ depending on $f|_U$ only.
    \end{enumerate}
\end{prop}

\begin{proof}
Without loss of generality, we may assume $N=1$. We have that $G = g \circ \Pi$ is a $C^{k+2}$ map with $\crit G = \Pi^{-1}(\crit g)$. Choose a smooth bump function $\chi \colon X \to \R$ with bounded $C^k$ norm and such that $\{\chi = 1\}$ contains a neighborhood of $C$ and is supported in $U$. Then $\chi \cdot G$ defines a $C^{k}$ function on $X$. We claim that, for sufficiently small $\delta>0$, the functions
     \[f_\delta(x) = f(x) + \delta \cdot \chi(x) \cdot G(x)\]
 satisfy the required conditions.

First, by the compactness of $C$, we may assume $\sup_{u \in U \setminus \{\chi =1 \}}(f(u) + \|Df(u)\|) <+\infty$. Moreover, by the Palais-Smale condition, we can also assume that $\inf_{u \in U \setminus \{\chi = 1\}} \|Df(u)\|>0$. Therefore, there exists $\delta_1>0$ such that
     \[\sup_{u \in U \setminus \{\chi =1 \}} \|\delta_1\cdot D(\chi G)(u)\| \ <\ \inf_{u \in U \setminus \{\chi = 1\}} \|Df(u)\|.\]
It follows that $Df_\delta(x) \neq 0$ in  $U \setminus \{\chi =1\}$ for $\delta \in (0,\delta_1)$. In $\{\chi =1\}$, by the Corollary of Lemma \ref{critical}, we have $Df(x)v \neq 0$ for any $x \in U \setminus C$ and any $v \in \ker D\Pi(x)$, so $f_\delta$ has no critical points in $\{\chi=1\} \setminus C$. Moreover, for any $x \in C$, we have $Df(x)=0$ and $\hc$ splits as $\ker D\Pi(x) \oplus T_xC$, where $D\Pi(x)=$ the orthogonal projection onto $T_xC$. Therefore $(\crit f_\delta) \cap C = \crit g$, and $(\crit f_\delta) \cap U = \crit g$.

 At any critical point $x\in C$ for $g$, if $\Pi^\perp$ is the projection onto $\ker D\Pi(x)$ relative to $T_xC \oplus \ker D\Pi(x)$. Then
     \[D^2f_\delta(x)(u,v) = D^2f(x)(\Pi^\perp u,\Pi^\perp v) + \delta \cdot D^2g(x)(D\Pi(x) u, D\Pi(x) v),\]
 proving the characterization of the Morse index $\ind(f_\delta,x_0)$ in (iii). Finally, (iv) follows from the construction, as $\chi$ is supported in $U$, and $g$, $\Pi$, and $\chi$ have bounded $C^k$ norms in $U$.
 %\textcolor{red}{what is $\omega, p,V$?}
 \end{proof}

\subsection{Heteroclinic orbits for Morse perturbations of Allen-Cahn energy in \texorpdfstring{$S^3$}{Lg}} \label{subsec:orbits} From now on, we write $\mathcal{H} = W^{1,2}(S^3)$, and $f=\frac{1}{\e}E_\e \colon \mathcal{H} \to \R$ for the Allen-Cahn energy
    \[f(u) = \int_{S^3} \frac{|\nabla u|^2}{2} + \frac{W(u)}{\e^2},\]
for a symmetric double-well potential $W \in C^\infty(\R)$ satisfying the conditions stated in Section \ref{subsec:doublewell}. Here $\e \in (0,\e_1)$ is fixed as mentioned in Proposition \ref{energy-mb}, and we continue to write $A_\e$ and $B_\e$ for the critical manifolds of nonconstant solutions at two lowest nontrivial critical energy levels.

We also fix a pair of critical points $\pm b_0 \in B_\e$ -- that is, the Allen-Cahn approximations of a Clifford torus --, and a pair of ground states $\pm a_0 \in A_\e$, and choose \emph{even} smooth Morse functions $g_1\colon A_\e \to (-\infty,0]$ with $g_1^{-1}(0) = \{ \pm a_0\}$ and $g_2 \colon B_\e \to [0,+\infty)$ with $g_2^{-1}(0) = \{ \pm b_0\}$. The existence of such functions follows directly from the existence of suitable Morse functions in the sphere (such as $g(x) = -x_1^2 - 2x_2^2 - 3x_3^2 - 4x_4^2$) and in the Grassmannian space $G(2,4)$ (such as $g(x) = x_1^2 + 2x_2^2 + 3x_3^2 + 4x_4^2 +5x_5^2 + 6x_6^2$, as we identify the oriented Grassmannian $G^+(2,4)$ with $S^2 \times S^2$).

By Proposition \ref{perturb}, there exist a symmetric tubular neighborhood $N_0$ of $A_\e \cup B_\e$, and $C^2$ \emph{even} perturbations $\{f_\delta\colon \hc \to \R\}_{\delta \in (0,\delta_1)}$ of $f$ such that:
    \begin{enumerate}
        \item[(i)] $f_\delta$ are Morse functions on $N_0$ and $f_\delta =f$ outside $N_0$;
        \item[(ii)] $f_\delta = f + \delta G$ on a smaller tubular neighborhood $N_1 \subset N_0$ of $A_\e \cup B_\e$, where $G=g_1 \circ \Pi$ near $A_\e$, $G = g_2 \circ \Pi$ near $B_\e$ and $\Pi$ is the $W^{1,2}$-nearest point projection from $N_0$ onto $A_\e \cup B_\e$;
        \item[(iii)] $(\crit f_\delta) \cap N_0 = \crit g_1 \, \cup\, \crit g_2$ (in particular, they are finite and independent of $\delta$);
        \item[(iv)] $\ind(u,f_\delta) = \ind(u,f) + \ind(u,g_1) = 1+\ind(u,g_1)$ for any $u \in \crit g_1$, and $\ind(u,f_\delta) = \ind(u,f) + \ind(u,g_2) = 5+\ind(u,g_2)$ for any $u \in \crit g_2$.
    \end{enumerate}
We observe that, by the choice of $g_1$ and $g_2$, $\pm a_0$ and $\pm b_0$ are critical points of $f_\delta$, with $\ind(\pm a_0,f_\delta) = 4$ and $\ind(\pm b_0,f_\delta) = 5$.

\begin{rmk}
Note that any bounded sequence $\{u_i\}$ of functions such that $Df_{\delta_i}(u_i) = 0$ with $\delta_i \to 0$ subconverges to some $u \in \crit f$ in $\mathcal{H}$. In fact, by passing to a subsequence, we may assume that either $u_i \in N_0$, $\forall i$, or $u_i \in \hc \setminus N_0$, $\forall i$. In the former case, we see that $u_i$ admits a further subsequence which is constant (by (iii) above). In the latter case, we get $Df(u_i) = 0$, $\forall i$, and by the Palais-Smale property for $f$, we conclude that $u_i$ subconverges to some $u \in \crit f$.
\end{rmk}

The proof of the \ref{main} relies on the following existence result for orbits of the Morse perturbations $f_\delta$:

\begin{thm} \label{flow1}
There exist $\delta_0>0$ (possibly depending on $\e>0$) such that, for every $\delta \in (0,\delta_0)$, there exists an entire solution $u_\delta \colon \R \to \mathcal{H}$ of the $L^2$-gradient flow $\partial_t u_\delta = -Df_\delta(u_\delta)$ of $f_\delta$ such that $\{u_\delta(t) \mid t \in \R, \delta >0\}$ is bounded in $X^\alpha = W^{2\alpha,2}(S^3)$, for $\alpha \in (\frac{3}{4},1)$ and
    \[\|u_\delta(-t) - b_0\|_{W^{1,2}(S^3)} \to 0 \qquad \text{and} \qquad \|u_\delta(t)-a_0\|_{W^{1,2}(S^3)} \to 0\]
as $t \to +\infty$.
\end{thm}
\noindent In the remainder of the present section, we will prove our \ref{main} assuming Theorem \ref{flow1}. The proof of Theorem \ref{flow1} will be the main goal of the next section. 

We begin with some auxiliary results about the perturbations $f_\delta$ and their derivatives. We write $\Xi = A_\e \cup B_\e$, and emphasize that $\e>0$ is fixed.

\begin{lem} \label{lem:est1}
There exists $C_1>0$ such that
    \[|\left\langle Df(u), DG(u) \right\rangle_{W^{-1,2} \times W^{1,2}}| \leq C_1\|u-\Pi(u)\|_{L^2(S^3)}^2\]
for every $u \in N_1$.
\end{lem}

\begin{proof}
First recall that $D\Pi(u)$ is the orthogonal projection, relative to the $W^{1,2}$ inner product, onto the tangent space to $\Xi$ at $\Pi(u)$, for all $u \in N_1$. We write $g$ to mean either $g_1$ or $g_2$ on the critical manifolds $A_\e$ or $B_\e$, respectively. If we consider an $W^{1,2}$-orthonormal basis $\{v_j\}_{j=1}^d$ of $T_{\Pi(u)}\Xi$, then
    \begin{align*}
        \left\langle DG(u),\vphi \right\rangle_{W^{-1,2}\times W^{1,2}} & = \left\langle Dg(\Pi(u)) \circ D\Pi(u), \vphi \right\rangle_{W^{-1,2}\times W^{1,2}}\\
        & =  Dg(\Pi(u))\left(\sum_{j=1}^d \left\langle \vphi,v_j \right\rangle_{W^{1,2}}v_j\right)\\
        & = \left\langle \sum_{j=1}^d (Dg(\Pi(u))v_j)v_j, \vphi \right\rangle_{W^{1,2}}.
    \end{align*}
for all $\varphi \in W^{1,2}(S^3)$. Since the functions $v_j$ are smooth, the function $\psi(u) = \sum_j (Dg(\Pi(u))v_j)v_j$ satisfies
    \[\|\psi(u)\|_{L^\infty(S^3)} \leq \sum_{j=1}^d |Dg(\Pi(u))v_j|\cdot \|v_j\|_{L^\infty(u)} \leq C_0\cdot d\cdot \|g\|_{C^1},\]
where $C_0 = \sup\{ \|\vphi\|_{L^\infty(S^3)} \mid \vphi \in \ker D^2f(v), \ v \in \Xi,\  \|\vphi\|_{W^{1,2}(S^3)} = 1\}$ depends only on the eigenfunctions in the kernel of the linearized $\e$-Allen-Cahn operator at solutions in $\Xi$. Since $\Pi(u)$ is a critical point of $f$ and $\psi(u) \in \ker D^2f(\Pi(u))$, we compute:
    \begin{align*}
            \left|\left\langle Df(u), DG(u) \right\rangle_{W^{-1,2} \times W^{1,2}}\right| &= \left| \int_{S^3} \left\langle \nabla \psi(u),\nabla u \right\rangle + \psi(u)\cdot \frac{W'(u)}{\e^2} \right|\\
            & = \left| \int_{S^3} \left\langle \nabla \psi(u),\nabla u - \nabla \Pi(u) \right\rangle + \psi(u)\cdot \frac{W'(u)-W'(\Pi(u))}{\e^2} \right|\\
            & \leq \frac{1}{\e^2} \int_{S^3} |\psi(u)|\cdot |r(u)|
    \end{align*}
where
    \[|r(u)| = |W'(u) - W'(\Pi(u)) - W''(\Pi(u))(u-\Pi(u))| \leq \kappa |u-\Pi(u)|^2\]
for some $\kappa>0$ depending only on the double-well potential $W$. Therefore,
    \[|\left\langle Df(u), DG(u) \right\rangle_{W^{-1,2} \times W^{1,2}}|\leq C_1 \|u-\Pi(u)\|_{L^2(S^3)}^2\]
for $C_1 = \kappa C_0 d\|g\|_{C^1}/ \e^2$.
\end{proof}
%\textcolor{red}{remove this formula for $C_1$ or add the term $d ||g||_{C^1}$}

The second result is a consequence of the Morse-Bott property for the energy $f=\frac{1}{\e}E_\e$.

\begin{lem} \label{lem:est2}
There exist a tubular neighborhood $N_2 \subset N_1$ of $\Xi$ and $\tau>0$ such that
    \[\tau \|u - \Pi(u)\|_{W^{1,2}(S^3)} \leq \|Df(u)\|_{W^{-1,2}(S^3)} \qquad \text{for all} \qquad u \in N_2.\]
\end{lem}

\begin{proof}
Consider, for each $\bar u \in \Xi$, the space
    \[V_{\bar u} = \{ \vphi \in W^{1,2}(S^3) \mid \left\langle \varphi, h \right\rangle_{W^{1,2}(S^3)} = 0, \ \forall h \in \ker D^2f(\bar u)\},\]
namely the orthogonal complement of the tangent space of $\Xi$ at $\bar u$, endowed with the $W^{1,2}(S^3)$ inner product. Define $z \colon V_{\bar u} \to \R$ by
    \[z(\vphi) = f(\bar u + \vphi).\]
Then $z$ is as differentiable as $f$ and its derivative is given by $Dz(\vphi) = Df(\bar u + \vphi)$. Moreover, for any $\xi \in V_{\bar u}$, we may consider $D^2z(\vphi)\xi$ as a linear operator $V_{\bar u} \to \R$ given by
    \[\left\langle D^2 z(\varphi)\xi, \psi \right\rangle = \left\langle D^2f(\bar u+ \varphi) \xi, \psi \right\rangle_{W^{-1,2} \times W^{1,2}} = \int_{S^3} \left\langle \nabla \xi, \nabla \psi \right\rangle + \frac{W''(\bar u+\varphi)}{\e^2} \xi \psi.\]
In particular, $0 \in V_{\bar u}$ is a critical point of $z$ and $D^2z(0) = Df^2(\bar u)$. If we let $Z=Dz\colon V_{\bar u} \subset W^{1,2}(S^3) \to W^{-1,2}(S^3)$, then $Z$ is differentiable, $Z(0) = 0$ and $DZ(0) = D^2f(\bar u)$.\medskip

\noindent\textbf{Claim.} There exists $\beta_{\bar u}>0$ such that $\|DZ(0)\vphi\|_{W^{-1,2}(S^3)} \geq \beta_{\bar u} \|\vphi\|_{W^{1,2}(S^3)}$ for all $\vphi \in V_{\bar u}$.\medskip

\noindent \emph{Proof.} Otherwise, there would exist a sequence $\{\vphi_j\}$ in $V_{\bar u}$ with $\|\vphi_j\|_{W^{1,2}(S^3)}=1$ and $\|DZ(0)\vphi_j\|_{W^{-1,2}(S^3)} \to 0$. By the Rellich-Kondrachov Theorem, there exists $\vphi \in W^{1,2}(S^3)$ such that $\vphi_j \to \vphi$ strongly in $L^2(S^3)$ and weakly in $W^{1,2}(S^3)$. On the other hand, by the self-adjointness of $DZ(0)$ and
    \[|\left\langle DZ(0)\vphi_j, \psi\right\rangle_{W^{-1,2}\times W^{1,2}}| \leq \|DZ(0)\vphi_j\|_{W^{-1,2}(S^3)} \|\psi\|_{W^{1,2}(S^3)} \to 0,\]
we see that $\vphi \in \ker DZ(0) = \ker D^2f(\bar u)$, which is the tangent space to $\Xi$ at $\bar u$. Since $\varphi \in V_{\bar u}$, we see that $\varphi = 0$ and
$\|\nabla \vphi_j\|_{L^2(S^3)} \to 1$. On the other hand, if $c>0$ is such that $W''\geq -c$, then
    \begin{align*}
        \|\nabla \vphi_j\|_{L^2(S^3)}^2 &\leq \frac{c}{\e^2} \|\vphi_j\|_{L^2(S^3)}^2 + \int |\nabla\vphi_j|^2 + \frac{W''(\bar u)}{\e^2}\vphi_j^2 \\
        & \leq \frac{c}{\e^2} \|\vphi_j\|_{L^2(S^3)}^2 + \left\langle DZ(0)\vphi_j,\vphi_j \right\rangle_{W^{-1,2} \times W^{1,2}} \to 0
    \end{align*}
 which yields a contradiction. This proves the claim. \hfill\qed

 Let $\rho_{\bar u}>0$ be such that
    \[\|\vphi\|_{W^{1,2}(S^3)} < \rho_{\bar u} \implies \|Z(\vphi) - DZ(0)\vphi\|_{W^{-1,2}(S^3)} \leq \frac{\beta_{\bar u}}{2}\|\vphi\|_{W^{1,2}(S^3)}.\]
If $\vphi \in V_{\bar u}$ and $\|\vphi\|_{W^{1,2}(S^3)}<\rho_{\bar u}$, we conclude that
\begin{equation*}
\begin{split}
\|Df(\bar u + \vphi)\|_{W^{-1,2}(S^3)} &= \|Z(\vphi)\|_{W^{-1,2}(S^3)} \\
&\geq \|DZ(0)\vphi\|_{{W^{-1,2}(S^3)}} - \frac{\beta_{\bar u}}{2} \|\vphi\|_{W^{1,2}(S^3)}\geq \frac{\beta_{\bar u}}{2}\|\vphi\|_{W^{1,2}(S^3)}.
\end{split}
\end{equation*}

Since the functional $f$ is equivariant with respect to the isometries of $S^3$ (which act transitively in $A_\e$ and $B_\e$), it follows that we may select $\beta_{\bar u}$ independently of $\bar u$, thus concluding the proof.\qedhere
\end{proof}

\begin{rmk}
We note that, alternatively, one may use the compactness of the critical manifolds $A_\e$ and $B_\e$ to obtain a constant $\beta_{\bar u}>0$ that is independent of the critical point $\bar u$, without resorting to the symmetry of $S^3$.
\end{rmk}

\begin{lem} \label{d_comparison}
Let $N_2$ be as in Lemma \ref{lem:est2}. There exists $\delta_2 \in (0,\delta_1]$ such that 
    \[\|Df(u)\|_{W^{-1,2}(S^3)}\leq 2\|Df_{\delta}(u)\|_{W^{-1,2}(S^3)}\]
for all $u \in N_2$ and for all $\delta \in (0,\delta_2)$
\end{lem}

\begin{proof}
Let $\delta_2 = \min\{1/(2C_1\tau^2),\delta_1\}$. For simplicity, write $\|\cdot\|=\|\cdot\|_{W^{-1,2}(S^3)}$. Then
    \begin{align*}
        \|Df_\delta(u)\| \|Df(u)\| & \geq \left| \left\langle Df(u), Df_\delta(u) \right\rangle_{W^{-1,2}(S^3)}\right|\\
        & = \left| \|Df(u)\|^2 + \delta \left\langle Df(u), DG(u) \right\rangle_{W^{-1,2}(S^3)} \right|\\
        & \geq \|Df(u)\|^2 - C_1\delta \|u-\Pi(u)\|^2_{W^{1,2}(S^3)} \qquad\quad \text{(By Lemma \ref{lem:est1})}\\
        & \geq \|Df(u)\|^2 - C_1 \tau^2 \delta \|Df(u)\|^2 \qquad\qquad \qquad \text{(By Lemma \ref{lem:est2})}\\
        & \geq  \frac{1}{2}\|Df(u)\|^2
    \end{align*}
for all $u \in N_1$ and $\delta \in (0,\delta_2)$ as claimed.
\end{proof}

\begin{proof}[Proof of \ref{main}]
%First note that each orbit $u_\delta$ lies in the unstable manifold of $b_0$ with respect to this semiflow. A neighborhood of $b_0$ in this unstable manifold is parametrized by the eigenfunctions of the self-adjoint operator associated to $D^2f_\delta(b_0)$ associated to its negative eigenvalues, as described in Theorem 5.2.1 \cite{Henry}. Hence, after possibly translating each $u_\delta$ so that $u_\delta(0)$ is in the boundary of $B_\rho(b_0)$, for a small $\rho>0$ not depending on $\delta$, we see that there is a sequence $\delta \downarrow 0$ (not labeled) such that $u_\delta(0)$ converges in $W^{1,2}$ to a function $u_0$ on the unstable manifold of $b_0$ with respect to negative $L^2$-gradient flow of $f=\frac{1}{\e}E_\e$. 
For every $\delta>0$ such that $\delta<\min\{\delta_0,\delta_2\}$ (where $\delta_0$ is given by Theorem \ref{flow1}), consider the corresponding entire solutions $\{u_\delta(t)\}_{t \in \R}$ of the semiflow generated by the autonomous equation $\partial_t u = -Df_\delta(u)$ in $X^\alpha=W^{2\alpha,2}(S^3) \subset W^{1,2}(S^3)$ for $\alpha\in (3/4,1)$. 

By the boundedness of solutions $u_\delta$ in $X^\alpha$ and by Theorem \ref{contdependence}, there exists an entire $W^{1,2}$-solution $u \colon \R \to W^{1,2}(S^3)$ of $\partial_t u = -Df(u)$ in $S^3$ such that, for any $T>0$,
    \[\sup_{t \in [-T,T]}\|u_\delta(t) - u(t)\|_{W^{1,2}(S^3)} \to 0 \quad \text{as} \quad \delta\downarrow 0.\]

We may also assume, by translating the solutions $u_\delta(t)$ in time, that $f_\delta(u_\delta(0))$ does not converge to $\omega_1(f)$ or $\omega_5(f)$. By parabolic regularity (see e.g. \cite[Example 3.6]{Henry}), we conclude that $u(t)$ is a nonconstant solution of \eqref{PAC} in $S^3$. Moreover, since $u_\delta(t) \to a_0$ in $W^{1,2}$ and $f_\delta(u_\delta(t))$ is strictly decreasing, as $t \to +\infty$, we have that $f_\delta(u_\delta(t)) \geq f_\delta(a_0) = f(a_0) + \delta G(a_0) = f(a_0)$ for all $t \in \R$, hence $f(u(t))$ is bounded from below away from $0$. Furthermore, we have $f_\delta(u_\delta(t)) < f_\delta(b_0)$ for all $t$ and hence $f(u(t)) \leq f(b_0) = \omega_5(\frac{1}{\e}E_\e)$ for all $t \in \R$. Therefore, Proposition \ref{compact_flow} implies that there exist an Allen-Cahn approximation $u^{-\infty} \in B_\e$ of a Clifford torus and a ground state solution $u^{+\infty} \in A_\e$ such that $\|u(\pm t)-u^{\pm \infty}\|_{W^{1,2}(S^3)} \to 0$ as $t \to \infty$. It suffices then to show that $u^{-\infty} = \pm b_0$ and $u^{+\infty} = \pm a_0$. We will prove the latter here, the proof of the former being completely analogous.\medskip

First, we note that $f_\delta(u_\delta(t))$ is continuous at $(t,\delta) = (+\infty,0)$, that is\medskip

\noindent\textbf{Claim.} For any $\eta>0$, there exist $\bar T>0$ and $\bar \delta \in (0,\delta_2)$ such that
    \[f_\delta(u_\delta(t)) - f_\delta(\pm a_0) < \eta, \qquad \text{for all} \qquad t \geq \bar T, \delta \in (0,\bar \delta).\]

\noindent\emph{Proof of the claim.} The proof follows closely \cite[Lemma 4]{CG2023}. If the claim does not hold, then there is $\eta>0$, a subsequence $\delta_j \downarrow 0$ and a sequence of real numbers $t_j \uparrow +\infty$ such that
    \[f_{\delta_j}(u_{\delta_j}(t_j)) - f_{\delta_j}(\pm a_0) \geq \eta \qquad \text{for all} \ j.\]
We have $f_\delta(\pm a_0) = f(\pm a_0) + \delta G(\pm a_0)$, so there exists $\delta'>0$ such that
    \[f_\delta(\pm a_0) > f(\pm a_0) -\frac{\eta}{3} \qquad \text{for all} \ \delta \in (0,\delta').\]
Let $T'>0$ be such that
    \[f(u(t)) < f(u^{+\infty}) + \frac{\eta}{3} = f(\pm a_0) + \frac{\eta}{3}, \qquad \text{for all} \ t \geq T'.\]
For this fixed $T'$, using $u_\delta(T') \to u(T')$ (in $W^{1,2}(S^3)$), we see that there exists $J \in \N$ such that
    \[f_{\delta_j}(u_{\delta_j}(T')) < f(u(T')) + \frac{\eta}{3}, \qquad \text{for all} \ j \geq J.\]
By taking a further subsequence, we may assume $t_J \geq T'$ and $\delta_J < \delta'$. But then, by the monotonicity of the energy along the flow, we get
    \begin{align*}
        f_{\delta_J}(\pm a_0) + \eta & \leq f_{\delta_J}(u_{\delta_J}(t_J)) \leq f_{\delta_J}(u_{\delta_J}(T'))\\
        & < f(u(T')) +\frac{\eta}{3} < \left(f(\pm a_0) + \frac{\eta}{3} \right) + \frac{\eta}{3} < \left(f_{\delta_J}(\pm a_0) + \frac{\eta}{3} \right) + \frac{2\eta}{3}.
    \end{align*}
which is a contradiction. \hfill\qed\bigskip

We are ready to show that $u^{+\infty}\in \{\pm a_0\}$. Suppose otherwise and let $\eta = G(\pm a_0) - G(u^{+\infty})$, which is positive by the choice of the Morse perturbation. Then there exist $\bar\delta>0$ and $T>0$ such that:
\begin{enumerate}
    \item[(i)] $|G(u(t)) - G(u^{+\infty})| < \eta/4$, for all $t \geq T$; and
    \item[(ii)] $f_\delta(u_\delta(t)) - f_\delta(\pm a_0) < \tau^2 \eta/(8C)$, for all $t \geq T$ and all $\delta \in (0,\bar \delta)$,
\end{enumerate}
where $C>0$ and $\tau$ are given by Lemmas \ref{lem:est1} and \ref{lem:est2}. By the convergence of $\{u_\delta(t)\}$ to $\{u(t)\}$ in compact intervals, we may also assume $|G(u_\delta(T)) - G(u(T))| < \eta/4$ for $\delta \in (0,\bar \delta)$. We estimate, for any $t \geq T$,
\begin{align*}
\int^t_{T} \frac{d}{ds} G(u_{\delta}(s)) ds & = \int_T^t \left\langle DG(u_\delta(s)), -Df_\delta(u_\delta(s))\right\rangle_{W^{-1,2}\times W^{1,2}}\,ds \\
& = \int_T^t \left[ - \left\langle Df(u_\delta(s)), DG(u_\delta(s)) \right\rangle_{W^{-1,2}\times W^{1,2}} - \delta\|DG(u_\delta(s)\|_{W^{-1,2}}^2 \right]\,ds\\
&\leq C \int_T^t\|u_\delta(s) - \Pi(u_\delta(s))\|_{W^{1,2}(S^3)}^2\,ds \qquad \qquad \qquad \text{(by Lemma \ref{lem:est1})}\\
&\leq \frac{C}{\tau^2} \int_T^t \|Df(u_\delta(s)\|_{W^{-1,2}}^2\,ds \qquad \qquad \qquad \qquad \qquad \text{(by Lemma \ref{lem:est2})}\\
&\leq \frac{2C}{\tau^2} \int_T^t \|Df_\delta(u_\delta(s))\|_{W^{-1,2}}^2\,ds \qquad \qquad \qquad \qquad \quad \text{(by Lemma \ref{d_comparison})}\\
&= \frac{2C}{\tau^2} \int_T^t \left(-\frac{d}{ds} f_\delta(u_\delta(s))\right)\,ds = \frac{2C}{\tau^2}\left(f_\delta(u_\delta(T)) - f_\delta(u_\delta(t)) \right)\\
&<\frac{2C}{\tau^2}(f_\delta(u_\delta(T)) - f_\delta(\pm a_0)) < \frac{\eta}{4}
\end{align*}
Hence, we conclude that
\begin{align*}
    G(u_\delta(t)) & < G(u_\delta(T)) + \frac{\eta}{4} < G(u(T)) + \frac{\eta}{2} < G(u^{+\infty}) + \frac{3\eta}{4} = G(a_0) - \frac{\eta}{4}
\end{align*}
which contradicts to the fact that $u_{\delta}(t) \to \pm a_0$ as $t \to +\infty$.

After possibly replacing the solution $u(t)$ by $(-u(t))$, this proves that there is an entire solution of \eqref{PAC} that converges in $W^{1,2}$ to $b_0$, as $t \to -\infty$, and to either $a_0$ or $(-a_0)$, as $t \to +\infty$. Since $b_0$ is invariant by the antipodal isometry $i(x)=-x$ in $S^3$, this finishes the proof of the \ref{main}, as we may replace $u(t)$ with $u(t)\circ i$ if necessary,
\end{proof}

We conclude this section with a formula for the derivative of the perturbed function $f_\delta$ that describes the nonlocal nonlinearity introduced by the perturbation term $\delta \chi(u)\cdot G(u)$. This will be used in the next section to describe the negative gradient (semi)flow of $f_\delta$.

\begin{lem} \label{lem:derivative}
There exist Lipschitz maps $\eta\colon W^{1,2}(S^3) \to W^{1,2}(S^3)$ and $\theta \colon W^{1,2}(S^3) \to \R$ such that
    \[\left\langle Df_\delta(u),v\right\rangle_{W^{-1,2}\times W^{1,2}} = (1+\delta\theta(u))\int_{S^3}\langle \nabla u,\nabla v \rangle + \frac{1}{\e^2}
    \int_{S^3}W'(u)v + \delta \chi(u) \int_{S^3} \eta(u)v.\]
\end{lem}

\begin{proof}
For simplicity we will write $\langle L,v\rangle$ for the natural pairing between $L \in W^{-1,2}(S^3)$ and $v \in W^{1,2}(S^3)$. First, we note that $Df(u) = D\left( \frac{1}{\e} E_\e\right)(u)$ can be written as $Df(u)v = -\Delta v + \frac{1}{\e^2}W'(u)v$, so it suffices to compute the derivative of $u \mapsto \chi(u) \cdot G(u)$. Let
    \[\rho(u) = \frac{1}{2}\mathrm{dist}_{W^{1,2}(S^3)}(u,\Xi)^2=\frac{1}{2}\|u - \Pi(u)\|_{W^{1,2}(S^3)}^2.\]
Note that $\rho$ is smooth on a tubular neighborhood of $\Xi$ and its derivative is given by
    \begin{align*}
        \left\langle D\rho(u),v \right\rangle & = \left\langle u-\Pi(u),v \right\rangle_{W^{1,2}(S^3)} = \int_{S^3}\langle \nabla (u-\Pi(u)), \nabla v \rangle+ (u-\Pi(u))v,
    \end{align*}
which means we can write $D\rho$ as $D\rho = -\Delta +\eta_1$, where $\eta_1$ is a bounded and Lipschitz map (with respect to the Sobolev norm) defined on a neighborhood of $\Xi$ by
    \[\eta_1(u) = u-\Pi(u)+\Delta(\Pi(u)) = u-\Pi(u) + \frac{W'(\Pi(u))}{\e^2},\]
where the last equality is a result of the fact that $\Pi(u)$ is a critical point of $f$. By the proof of Proposition \ref{perturb}, we may assume that $\chi(u) = \bar \chi(\rho(u))$ for a smooth function $\rho \in C^\infty(\R)$ that is $\bar\chi\equiv 1$ in $(-\infty,r']$ and $\bar\chi \equiv 0$ in $[r,+\infty)$. Then
    \[D\chi(u) = \bar\chi'(\rho(u))D\rho(u) = \bar\chi'(\rho(u))(-\Delta u + \eta_1(u))\]
and we can compute
    \[D(\bar \chi G)(u) = -\theta(u)\Delta(u) + \eta(u),\]
where $\theta(u) = G(u) \bar\chi'(\rho(u))$ and
    \begin{align*}
        \eta(u) &= G(u)\bar\chi'(\rho(u))\eta_1(u) + \chi(u)DG(u)\\
        & = G(u)\bar\chi'(\rho(u))\eta_1(u) + \chi(u)\left(1-\frac{1}{\e^2}W''(\Pi(u))\right) \nabla^{W^{1,2}}g(\Pi(u)),
    \end{align*}
Here we used the computation of $DG(u)$ from the proof of Lemma \ref{lem:est1} and the fact that functions in $T_{\Pi(u)}\Xi$ are solutions of the linearized Allen-Cahn equation at $\Pi(u)$, and denoted by $\nabla^{W^{1,2}}g$ the gradient of the Morse function $g \colon \Xi \to \R$ with respect to the $W^{1,2}(S^3)$ inner product.
\end{proof}

\section{topological argument} \label{sec: topological argument}

The primary objective of this section is to construct appropriate eternal solutions of the negative $L^2$-gradient flow of the perturbations $f_\delta$ of the Allen-Cahn energy $f_0=\frac{1}{\e}E_\e$, thus establishing the proof of Theorem \ref{flow1}.  We will use topological tools to study this flow using the variational theory introduced in Subsections \ref{subsec:minmax}, \ref{subsec:morse}, and \ref{subsec:orbits}. We continue to write $\hc = W^{1,2}(S^3)$.

We begin by proving some useful lemmas about the min-max values of $f_\delta$. First, we compare $p$-sweepouts, even maps into $\hc$, and homology classes in $H_{*}(\tilde{\hc},\Z_2)$, where $\tilde{\hc}$ is the quotient of $\hc \setminus \{0\}$ by the (free) $\Z_2$ action given by the antipodal map. For a $\mathbb{Z}_2$ equivariant map $h: Q \to \hc \setminus \{0\}$, we will denote by $\widetilde{h}: \tilde Q \to \widetilde{\hc}$ the induced map, where $\tilde Q$ is the quotient space obtained from $Q$ by the corresponding $\Z_2$ action.

\begin{lem} \label{lem:sweep-homology}
Let $Q \subset [-1,1]^N$, $N \in \mathbb{N}$, be a symmetric $p$-dimensional cubical subcomplex, and let $\gamma: Q \to \hc \setminus \{0\}$ be an even continuous map. Then $\ind_{\mathbb{Z}_2}(Q) \geq p$ if, and only if, there exists a symmetric $p$-dim subcomplex $K \subset Q$ such that $\partial K = 0$ and $[\widetilde{\gamma |_K}] \neq 0$ in $H_p(\tilde{\hc}, \mathbb{Z}_2)$.
\end{lem}

\begin{proof}
Recall that $\tilde{\hc}$ is a $K(\Z_2,1)$-Eilenberg--MacLane space, whose $\Z_2$-cohomology ring is generated by some $\lambda \in H^{1} (\tilde{\hc}, \mathbb{Z}_2)$. Thus, we see that $\ind_{\Z_2}(Q) \geq p$ if, and only if, $ \widetilde{\gamma}^* \lambda^p \neq 0$ in $H^p(\widetilde{Q}, \Z_2)$, that is, if and only if there exists $\tau \in H_p(\widetilde{Q}, \Z_2)$ such that $\left\langle\widetilde{\gamma}^* \lambda^p, \tau\right\rangle \neq 0$. By \cite[Chapter 6]{Massey}, the class $\tau$ can be represented by $[\tilde K]$, where $K$ is a symmetric subcomplex of $Q$ with $\partial K = 0$. Denoting by $i_K \colon K \to Q$ the inclusion map, we see that
\[\left\langle\widetilde{\gamma}^*\lambda^p, [\widetilde{K}]\right\rangle = \left\langle\widetilde{\gamma}^*\lambda^p, (i_{\widetilde{K}})_* [\widetilde{K}]\right\rangle = \left\langle i_{\widetilde{K}}^* \widetilde{\gamma}^* \lambda^p, [\widetilde{K}]\right\rangle = \left\langle\lambda^p, (\widetilde{\gamma |_{K}})_*[\widetilde{K}]\right\rangle,\]
proving the claimed result.
\end{proof}

For the next lemma, we use the notation introduced in the previous Section in the construction of $f_\delta$. We can characterize its first five widths using the classification result, Theorem \ref{lowenergy}, and the choice of the Morse function along the critical manifolds $A_\e$ and $B_\e$:

\begin{lem}
For sufficiently small $\delta > 0$, we have:
\begin{enumerate}
    \item[(i)] the lowest five critical levels of $f_\delta$ are $0=f_\delta(\pm 1)$, and the min-max levels $\omega_1(f_\delta)<\omega_2(f_\delta)<\omega_3(f_\delta)<\omega_4(f_\delta)$. In addition, for $p=1,2,3,4$, the critical levels $\crit f_\delta \cap f_\delta^{-1}(\omega_p(f_\delta))$ are precisely the (pairs of) critical points of the Morse function $g_1$ along the space of ground states, in increasing order of energy;
    \item[(ii)] $\crit f_{\delta} \cap f_{\delta}^{-1}(\omega_5(f_{\delta})) = \{\pm b_0\}$; and
    \item [(iii)] $\crit f_{\delta} \cap f_\delta^{-1}(\omega_4(f_{\delta}), \omega_5(f_{\delta})) = \emptyset$.
\end{enumerate}
In particular, $\omega_4(f_\delta)=f_\delta(a_0) = f(a_0) = \omega_4(f)$, $\omega_5(f_\delta) = f_\delta(b_0) = f(b_0) = \omega_5(f)$, and $\ind(u,f_\delta) \leq 3$ for any $u \in \crit f_\delta$ with $f_\delta(u)<\omega_4(f_\delta)$.
\end{lem}

\begin{proof}
We claim that there exists $\delta_3 > 0$ such that, for any $\delta \in (0,\delta_3)$, any nonconstant critical point $u$ of $f_\delta$ with $f_\delta(u) \leq \omega_4(f_\delta)$ is in the tubular neighborhood $N_2$ of $\Xi=A_\e \cup B_\e$. Otherwise, we can find a sequence $\delta_i \downarrow 0$ and $u_i \in \crit f_{\delta_i}$ with $f_{\delta_i}(u_i) \leq \omega_4(f_{\delta_i})$ such that $u_i \notin N_2$. By the Remark in Subsection \ref{subsec:orbits}, $\{u_i\}$ subconverges to some $u \in \crit f$ with $f(u) = \lim f_{\delta_i}(u_i) \leq \omega_4(f)$, where we used Lemma \ref{lem:convwidth}. By Theorem \ref{lowenergy}, we get $u \in A_\e \cup \{\pm 1\}$. Since $u_i$ are nonconstant, $f=f_\delta$ in a neighborhood of the constant functions $\pm 1$ and these are isolated critical points of $f$, we see that $u \in A_\e$. This implies $u_i \in N_2$ for large $i$ and we reach a contradiction.

Consequently, we get $\{ u \in \crit f_\delta \mid f_\delta(u) \leq \omega_4(f_\delta)\} \subset \{\pm 1\} \cup (\crit g_1)$ for $\delta \in (0,\delta_3)$, and $f_\delta$ has at most five critical levels. Since $\crit g_1$ is discrete, by \cite[Theorem 10.8]{Ghoussoub} we see that
\begin{align*}
\omega_1(f_{\delta}) < \omega_2(f_{\delta}) < \omega_3(f_{\delta}) < \omega_4(f_{\delta}).
\end{align*}
This proves (i), and implies $\omega_4(f_\delta) = \max f_\delta(\crit g_1) = f_\delta(a_0)=f(a_0)$, in addition to
    \[\ind(f_\delta,u) = 1 + \ind(g_1,u) \leq 3, \quad \text{for any} \ u \in \crit f_\delta \ \text{with} \ u \neq \pm a_0.\]
A similar argument (invoking the characterization of the fifth min-max critical level $\omega_5(f)$ in Theorem \ref{lowenergy}) shows that, for sufficiently small $\delta$, the functions $\pm b_0$ are the only critical points of $f_{\delta}$ with $f_\delta = \omega_5(f_\delta)$.

Finally, if $u_i \in \crit f_{\delta}$, $\delta_i \downarrow 0$, is a sequence with $\omega_4(f_{\delta_i}) < f_{\delta_i}(u_i) < \omega_5(f_{\delta}))$, then $u_i$ subconverges to some $u \in \crit f$ with $\omega_4(f) \leq f(u) \leq \omega_5(f)$. The rigidity of min-max values guarantees that $(\crit f) \cap f^{-1}(\omega_4(f), \omega_5(f)) = \emptyset$, hence $f(u)=\omega_p(f)$ for $p=4$ or $p=5$, and, by the argument above, $u_i \in N_1$ for large $i$. Since this implies $u_i \in \crit g_1$ or $u_i = \pm b_0$, we conclude that such a sequence of critical points of $f_{\delta_i}$ cannot exist, proving (iii).
\end{proof}

For the remainder of this section, we will write $\omega_p =\omega_p(f) = \omega_p(f_\delta)$, for $p=4,5$. By the results of R. Palais (see Section \ref{subsec:morse}), there exists $r \in (0,\omega_5 - \omega_4)$ such that the sublevel $(f_{\delta})^{\omega_5 + r}=\{f_\delta \leq \omega_5 +r\} \cap \hc$ is $C^2$ isomorphic to the sublevel $(f_{\delta})^{\omega_5-r}$ with two handles of type $(5, \infty)$ disjointly attached. Moreover, the same sublevel $(f_{\delta})^{\omega_5 + r}$ has
\[(f_{\delta})^{\omega_5 - r} \cup (\phi^+(D_0) \cup \phi^-(D_0))\]
as a deformation retract, where $D_0 = D^5 \times \{0\}$ and $\phi^{\pm}$ are the handle-attaching maps. By the excision theorem, it follows that
\[H_{*} ((f_{\delta})^{\omega_5 + r}, (f_{\delta})^{\omega_5 - r};\Z_2) \simeq H_{*}((f_{\delta})^{\omega_5 - r} \cup \phi^+(D_0) \cup \phi^-(D_0), (f_{\delta})^{\omega_5 - r};\Z_2)\]
has the homology classes of $\phi^{\pm}(D_0)$ as generators. From now on, we will write $\Omega:= \phi^+(D_0) \cup \phi^-(D_0)$.

It follows from Lemma \ref{lem:sweep-homology} and \cite[Lemma 6.2]{GasparGuaraco} that we can find a pulled-tight sweepout $A \in \mathcal{F}_5$ such that $A$ is the image of a $5$-dimensional cubical complex $K \subset [-1,1]^N$ with $\partial K = 0$ (by an odd continuous map into $\hc\setminus \{0\}$), and $\sup f(A) \leq \omega_5(f) + r$. We claim that $\widetilde{A}$ is homologous to the quotient $\tilde \Omega$ of $\Omega$ by the antipodal map, relative to the quotient of the sublevel $(f_{\delta})^{\omega_5 - r}$.   

To see this, note that $H_5(\widetilde{(f_{\delta})^{\omega_5+r}}, \widetilde{(f_{\delta})^{\omega_5-r}} / \sim;\Z_2)$ is generated by $[\widetilde{\Omega}]$. If $[\widetilde{A}] \neq [\widetilde{\Omega}]$, then the relative class $[\widetilde{A}]$ vanishes, that is, there exist a $5$-cycle $A_{\partial}$ in $(f_{\delta})^{\omega_5 - r}$ and a $6$-cycle $Z$ in $(f_{\delta})^{\omega_5 + r}$ such that $\partial Z = \tilde A + \tilde A_{\partial}$. Consequently $A_{\partial}$ is $\Z_2$-homologous to $A$ . By Lemma \ref{lem:sweep-homology}, this shows that (the support of) $A_{\partial}$ is a $5$-sweepout, thus 
\[\omega_5(f_{\delta}) \leq \sup f_{\delta}(A_{\partial}) \leq \omega_5(f_{\delta}) - r,\]
which is a contradiction. For the record, this implies that there exist a $5$-cycle $A_{\partial}$ in $(f_{\delta})^{\omega_5 - r}$ and a $6$-chain $Z_1$ in $(f_{\delta})^{\omega_5 + r}$ such that
    \begin{equation} \label{hom1}
        \partial Z_1 = A + A_\partial + \Omega.
    \end{equation}
    
Now consider the local unstable manifolds $\mathcal{W}^u(\pm b_0) \subset X^\alpha= W^{2\alpha,2}(S^3)$ for the stationary points $\pm b_0$ of the $L^2$-gradient semiflow of $-Df_\delta$, namely the semiflow described by the quasilinear autonomous equation
\begin{equation} \label{semiflow equation}
\frac{\mathrm{d}u}{\mathrm{d}t} = -Df_\delta(u) = (1+\delta\theta(u))\Delta u -\frac{W'(u)}{\e^2} - \delta\chi(u)\eta(u),
\end{equation}
where we use the notation from Proposition \ref{perturb} and the computation of the derivative of the functional $f_\delta$ derived in Lemma \ref{lem:derivative}. We remark that $\theta(u)$ vanishes on a neighborhood of the critical manifolds $A_\e$ and $B_\e$, and hence the equation is semilinear near the critical points $\pm b_0$.
By the unstable manifold Theorem -- see \cite[Theorem 5.2.1]{Henry} --, there is a local parametrization of the unstable manifold around $\pm b_0$ by $5$-dimensional disks $D_\pm \subset \mathcal{W}^u(\pm b_0)$ which are graphs over disks in the negative eigenspaces of $D^2f(\pm b_0)$. In particular, by shrinking $r$ and the disks in the description of the attaching maps $\phi^{\pm}$ if necessary, we can assume that $\Omega$ and $(D_{+}+D_-)$ are relatively homologous in $((f_\delta)^{\omega_5+r},(f_\delta)^{\omega_5-r})$. This means that there exist a symmetric $5$-chain $D_\partial$ in $(f_\delta)^{\omega_5-r}$ and a symmetric $6$-chain $Z_2$ in $(f_{\delta})^{\omega_5 + r}$ such that
    \begin{equation} \label{hom2}
        \partial Z_2 = (D_+ + D_-) + \Omega + D_{\partial}.
    \end{equation}

Let $\Psi$ be the semiflow generated by the evolution equation \ref{semiflow equation}. As noted in Theorem \ref{contdependence}, $\Psi$ is a Lipschitz map in an open subset of $\R_{\geq 0} \times X^\alpha$ into $X^\alpha$. The restriction of this map to $R_{\geq 0} \times \mathcal{W}^u(\pm b_0)$ can be extended to every $t<0$, as the solution of \ref{semiflow equation} with initial condition $y$ in the unstable manifold is defined for all $t<0$. Furthermore, 

\noindent \textbf{Claim.} If $y \in D_\pm$, then $t \mapsto \Psi(t,y)$ is defined for all $t \in \R$ and $\Psi(t,y)$ converges to a critical point of $f_\delta$ in $X^\alpha$ as $t \to +\infty$.

\begin{proof}[Proof of the claim]
It follows from \cite{Amann} that $u(t):=\Psi(t,y)$ is a solution of the integral equation
\[u(t) = U_{P(u(t))}(t,0)y - \int_0^tU_{P(u(t))}(t,s)\left(\frac{1}{\e^2}W'(u(s)) + \delta\chi(u(s))\eta(u(s))\right)\,ds,\]
where $U_{P(u(t))}$ denotes a \emph{parabolic fundamental solution} for the family of operators $\{P(u(t))\}_{t \in [0,T]}$ which are given, for any $T < t^+(y)$, by
    \[P(u(t))=-(1+\delta \theta(u(t)))\Delta,\]
as operators with domain $\mathcal{D}(P(u(t)) = W^{2,2}(S^3)$ into $L^2(S^3)$. We note that we may assume $D_{\pm}$ is contained in a sufficiently small $W^{2\alpha,2}(S^3)$-neighborhood of $\pm b_0$ so that, by the Sobolev embedding, $\|y\|_{L^\infty(S^3)}<1$ for any $y \in D_{\pm}$. Since $\theta \equiv 0$ and $\eta \equiv 0$ outside of a $W^{1,2}(S^3)$-neighborhood of $\Xi$, we see (using the maximum principle for \eqref{PAC}) that $\Psi(t,y)$ is bounded in $W^{1,2}(S^3)$ in its domain of existence. Furthermore, since $W'$ grows linearly (by the asumption (Q2)) and $\delta\chi(u)\eta(u)$ is Lipschitz in $W^{1,2}(S^3)$, we conclude that $\left\|{\textstyle\frac{1}{\e^2}}W'(u(t)) + \eta(u(t))\right\|_{L^2(S^3)}$ is bounded in the domain of $u$.

Now we can use \cite[Theorem 5.1]{Amann} and argue as in Corollary 3.3.5 \cite{Henry} to show that $\|u(t)\|_{X^\alpha}$ is bounded. The long time existence follows then from Theorem \ref{existence} (replacing $\alpha$ and $\alpha'$ by $\frac{1}{2}$ and $\alpha$, respectively). 
\end{proof}

For each $t \geq 0$, write
\begin{align*}
    D_{\pm}^t &= \{\Psi(t,y) \mid y \in D_{\pm}\} , \qquad & S_{\pm}^t & = \{\Psi(t,y) \mid y \in \partial D_{\pm}\}\\
    \widetilde{D}^t & = \widetilde{D_{\pm}^t}, \qquad & \widetilde{S^t} & = \widetilde{S_{\pm}^t},
\end{align*}
By the Claim above and by Theorem \ref{existence}, these are well defined. Moreover, since $f_\delta$ is a Morse functional below the critical level $\omega_5(f_\delta)$, for each $a \in D_{\pm}$, the orbit $\psi^t (a)$ converges, as $t \to + \infty$, to either a critical point of the Morse function $g_1$ along $A_\e$, or to $\pm 1$.

\begin{proof}[Proof of Theorem \ref{flow1}]

It suffices to show that there exists $a \in \partial (D_{\pm})$ such that $\Psi(t,a) \to \pm a_0$ as $t \to  \infty$. We argue by contradiction, assuming that there does not exist $a \in \partial D_{\pm}$ such that $\Psi(t,a) \to \pm a_0$ as $t \to  \infty$.

Let $l \in (\omega_3(f_{\delta}), \omega_4(f_{\delta}))$. Then, for any $a \in \partial (D_{\pm})$ there exist $T_a > 0$ such that 
\[f_\delta\left(\Psi(t,a)\right) <l,\quad \text{for all} \quad t \geq T_a.\]
By the compactness of $\partial D_{\pm}$ and the continuity of $\Psi$, it follows that there exists $T_l > 0$ such that $S_{\pm}^t \subset (f_{\delta})^l$, for all $t \geq T_l$.

Since all critical points in $(f_\delta)^l$ have index $\leq 3$, by \cite{Palais}, we get $H_4(\widetilde{(f_{\delta})^l};\Z_2) = 0$, and there exists a $5$-cycle $\tilde B_0$ in $\widetilde{(f_{\delta})^l}$ such that $\partial \widetilde{B_0} = \widetilde{S}^{T_l}$. If we let $B_0 \subset (f_\delta)^l \subset (f_\delta)^{\omega_4}$ be the lift of $\widetilde{B_0}$ to $\hc$, then we see that $B := D_+^{T_l} \cup D_-^{T_l} \cup B_0$ is a $5$-cycle in $(f_\delta)^{\omega_5}$ and 
\[\ind_{\Z_2}(B) \leq \ind_{\Z_2} (B_0) + 1 \leq \ind_{\Z_2} (\{f_{\delta} < l\}) + 1\leq 5.\]
where we used that $D_+^t \cup D_-^t$ has $\ind_{\Z_2} = 1$ for $t \in [0,T_l]$ (we may assume that the local unstable manifolds of $\pm b_0$ are disjoint). In conclusion, $B \notin \mathcal{F}_5$ (recall that the sets in $\mathcal{F}_5$ have index $\geq 6$).  By Lemma \ref{lem:sweep-homology}, this means that there exists a $6$-chain $Z_3$ in $(f_{\delta})^{\omega_5 + r}$ such that 
    \begin{equation} \label{hom3}
    \partial Z_3 = B  = (D_+ + D_-) + C + B_0,
    \end{equation}
where we denote by $C$ the $5$-chain $(D_+^{T_l} - D_+) + (D_-^{T_l} - D_-)$ supported in $(f_\delta)^{\omega_5-r}$, so that $B = (D_+ + D_-) + C + B_0$.

By combining \eqref{hom1}, \eqref{hom2}, and \eqref{hom3}, we see that the $5$-cycle
\[K := A_{\partial} + D_{\partial} + C + B_0\]
satisfies
\[\partial (Z_1 + Z_2 + Z_3) = A + K,\]
namely, $A$ and $K$ are $\Z_2$-homologous in $\omega_5(f)+r$. Using Lemma \ref{lem:sweep-homology} once again, we conclude that $K$ is a $5$-sweepout (recall that $A \in \mathcal{F}_5$). But this yields a contradiction, as $K$ is supported in $(f_{\delta})^{\omega_5 - r}$.
\end{proof}

\bibliographystyle{acm}{
\bibliography{main}}

\end{document}